\documentclass[final]{siamltex}
\usepackage{epsfig}
\usepackage{amsmath,amssymb}
\usepackage {graphicx}
\usepackage{amsmath,amssymb}

\newtheorem{rem}{Remark}[section]
\newtheorem{algorithm}{Algorithm}

\title{Preconditioning the Restarted and Shifted Block FOM
Algorithm for Matrix Exponential Computation}

\author{Gang Wu\thanks{Corresponding author (G. Wu). Department of Mathematics,
China University of Mining and Technology, Xuzhou, 221116, P. R. China.
E-mail: {\tt wugangzy@gmail.com} and {\tt gangwu76@126.com}. This author is
supported by the National Science Foundation of China under grant 11371176, the Natural Science Foundation of Jiangsu
Province under grant BK20131126, the 333 Project of Jiangsu Province, and the Talent Introduction Program of China
University of Mining and Technology.}
        \and Hong-kui Pang\thanks{School of Mathematics and Statistics, Jiangsu Normal University, Xuzhou, 221116,
        Jiangsu, P. R. China. E-mail: {\tt panghongkui@163.com}. This author is supported by the National
Science Foundation of China under grant 11201192, the
Natural Science Foundation of Jiangsu Province under grant
BK2012577, and the Natural Science Foundation for Colleges and
Universities in Jiangsu Province under grant 12KJB110004.}
\and Jiang-Li Sun
\thanks{School of
Mathematics and Statistics, Jiangsu Normal University, Xuzhou, 221116, Jiangsu, P. R. China.} E-mail: {\tt sunjiangli1108@163.com}.}

\begin{document}

\maketitle

\begin{abstract}
The approximation of $e^{tA}B$ where $A$ is a large sparse matrix and $B$ a rectangular matrix is
the key ingredient in many scientific and engineering computations. A powerful tool to manage the matrix exponential function
is to resort to a suitable rational approximation such as the Carath$\acute{\rm e}$odory-Fej$\acute{\rm e}$r approximation,
whose core reduces to solve shifted linear
systems with multiple right-hand sides. The restarted and shifted block FOM algorithm is a commonly used technique for this problem.
However, determining good preconditioners for shifted systems that preserve the original
structure is a difficult task.
In this paper, we propose a new preconditioner for the restarted and shifted block FOM algorithm.
The key is that the absolute values of the poles of
the Carath$\acute{\rm e}$odory-Fej$\acute{\rm e}$r approximation are medium sized and can be much smaller than the norm of the matrix in question.
The advantages of the proposed strategy are that we can precondition all the shifted linear systems simultaneously, and preserve the original
structure of the shifted linear systems after restarting. Theoretical results are provided to show the rationality of our preconditioning strategy.
Applications of the new approach to Toeplitz matrix exponential problem are also discussed.
Numerical experiments illustrate the superiority of the new algorithm over many state-of-the-art algorithms for matrix exponential.
\end{abstract}



\pagestyle{myheadings} \thispagestyle{plain} \markboth{G. WU, H. PANG AND J. SUN}{PRECONDITIONING SHIFTED BLOCK FOM
FOR MATRIX EXPONENTIAL}

\section{Introduction}

\setcounter{equation}{0}

The computation of exponential function of large and sparse matrices is now ubiquitous in numerical analysis
research \cite{Higham,ML2}.
%However, even if $A$ is sparse, the matrix exponential $e^A$ is full in general, and standard strategies
%may fail because of huge computational costs or numerical stability issues \cite{Higham,ML2}. In many applications,
%rather than $e^A$, one is required to compute $e^AB$ for a given rectangular matrix $B$ \cite{AH,ABK,AB2,LS2}.
In this paper, we pay special attention to numerical approximation
of the matrix exponential to a block vector
\begin{equation}\label{1}
Z(t)=e^{tA}B,
\end{equation}
where $t\in\mathbb{R}$ is a fixed constant, $tA\in
\mathbb{R}^{n\times n}$ is an $n$-by-$n$ large sparse, negative definite matrix with $\|tA\|\gg 1$, and
$B=\big[{\bf b}_{1},{\bf b}_{2},\ldots,{\bf b}_{p}\big]\in \mathbb{R}^{n\times p}$ is
a block vector with $1\leq p\ll n$. That is, the real part of the spectrum of $tA$ lies in the left half plane.
This problem is the key ingredient of the computation of exponential integrators \cite{AH,CM,HO}, time-dependent
partial differential equations \cite{GS}, the
approximation of dynamical systems \cite{Dyn}, and so on \cite{AB2,LS2,TGB}.
For notation simplicity, we denote $tA$ by $A$ from now on unless otherwise stated.

%However, this type of
%methods may suffer from inefficiency because of slow
%convergence \cite{EE,EEG,NW,Sidje}.
% %Recently, Al-Mohy and Higham \cite{AH} proposed an
%algorithm for computing (\ref{1}) without explicitly forming $e^A$. It uses the scaling part of
%the scaling and squaring method together with a truncated Taylor
%series approximation to the exponential. Specifically, when $A$ is Hamiltonian, skew-Hamiltonian
%or skew-symmetric, some structure preserving methods were investigated in \cite{ABK,AB2,LS2}.

The Krylov subspace methods are widely used for this type of problem, which
work on the subspaces generated by $A$ and ${\bf b}_i~(i=1,2,\ldots,p)$.
Generally speaking, there are two classes of Krylov subspace methods for evaluating (\ref{1})
when $A$ is large and sparse \cite{PS2}. In the first class of methods, the matrix is projected into
a much smaller subspace, then the exponential is applied to the reduced matrix,
and finally the approximation is projected back to the original large space \cite{residual,DK1,DK2,EE,MN,Saad2,vH,Ye}.
In the second class of methods, the exponential function is first approximated
by simpler functions such as rational functions,
and then the action of the matrix exponential is
evaluated \cite{FKL,FS,LS,Tre1,Tre2,WW}. In this case, the core of (\ref{1}) reduces to solving the following linear
systems with multiple shifts and multiple right-hand sides
\begin{eqnarray}\label{3}
(A-\tau_{i}I)X_{i}=B, \quad i=1,2,\ldots,\nu,
\end{eqnarray}
where $\tau_{i}\in \mathbb{C}~(i=1,2,\ldots,\nu)$ are the shifts. Throughout this paper, we make the assumption that both $A$
and the shifted matrices $A-\tau_{i}I~(i=1,2,\ldots,\nu)$ are
nonsingular. Note that there are $\nu\times p$
linear systems altogether.

The use of rational functions and their
partial fraction expansion allows us to exploit and generalize known
properties of Krylov subspace methods for the solution of algebraic
linear systems \cite{Saad,Simoncini3}. For instance, the first way for solving (\ref{3}) is to evaluate the $p$ shifted
linear systems separately using some shifted Krylov subspace algorithms
\cite{Frommer,FG,GZL,Simoncini,SSX}.
In \cite{DMW}, Darnnu {\it et al.} proposed a deflated and shifted GMRES algorithm for systems with multiple shifts
and multiple right-hand sides. In
this algorithm, the $p$ shifted linear systems are solved separately, and
eigenvector information from solution of the first right-hand side is
utilized to assist the convergence of the subsequent ones. However, one
has to solve the shifted systems with an auxiliary right-hand side,
which is unfavorable for very large matrices.
The second way for solving (\ref{3}) is to calculate
the $\nu$ linear systems with multiple right-hand sides sequentially,
using some block Krylov subspace methods such as the block GMRES algorithm with deflation \cite{Morgan}.
However, one has to solve the $\nu$ linear systems with multiple right-hand sides one by one,
and this algorithm may suffer from the drawback of slow convergence \cite{Morgan}.

The third way is to use the
shifted block GMRES algorithm \cite{WWJ}. In essence, it is a
generalization of the shifted GMRES algorithm \cite{FG} to its block version.
There are some deficiencies in the shifted block GMRES algorithm for linear systems with complex shifts.
Firstly, when both $A$ and $B$ are real while the shifts
$\tau_i~(i=1,2,\ldots,\nu)$ are complex, the initial block vector for the
next cycle after restarting is also complex. This problem can arise in
projection methods for rational approximation to the matrix
exponential \cite{LS,WW}. Consequently, the expensive steps for
constructing the orthogonal basis have to be performed in complex arithmetic
after the first cycle. Secondly, the shifted block GMRES algorithm may suffer from the drawbacks of ``stagnation" and ``near breakdown" in practice \cite{WWJ}. Thirdly, one has to pick the ``seed" block
system in advance. More precisely, we have to determine $\tau_1$ in
advance, which is a difficult task if there is no
further knowledge on the shifted block linear systems {\it a priori}. Finally, in the shifted block GMRES algorithm,
only the ``seed" block system residual is minimized, whereas
no minimization property is satisfied by the residuals of the
``additional" block systems \cite{FG,WWJ}.

In order to overcome these difficulties, we first introduce a
shifted block Full Orthogonalization Method (FOM) with deflated restarting. It is a block and deflated version of the
shifted FOM algorithm due to Simoncini
\cite{Simoncini}.
However, the algorithm may
converge very slowly or even may fail to converge in practice.
%Preconditioning
%techniques can reduce the computational effort of the shifted Krylov subspace
%algorithms, and they are essential to achieve convergence in many cases
%\cite{Je,Simoncini3,WWJ}.
A commonly used technique for accelerating
iterative solvers is preconditioning \cite{Saad,Simoncini3}.
In general, the reliability of iterative
techniques, when dealing with various applications, depends much
more on the quality of the preconditioner than on the particular
Krylov subspace methods \cite{Saad}. Therefore, how to construct an efficient preconditioner
is a crucial task for shifted linear systems with multiple right-hand sides \cite{DMW,Je,Simoncini3}.
%Therefore, one remedy is to employe some preconditioning
%techniques for improving efficiency of the shifted block FOM algorithm.

The study of preconditioners for shifted
linear systems or sequences of linear systems is an active area, and many
technologies have been proposed when $A$ is symmetric positive
definite (SPD), complex symmetric or Hermitian; see
\cite{Benzi,Bel,Ber,Freund,Su} and the references therein.
%proposed an approximate inverse preconditioning
%strategy for preconditioning the shifted linear systems; Bellavia {\it et al.} \cite{Bel} proposed a technique for building
%effective and low cost preconditioners for sequences of shifted linear systems.
%In \cite{Su}, Su\'{a}rez {\it et al.} discussed updating incomplete factorization preconditioners for shifted linear systems
%of the
%form  $(M+\xi(\gamma) N){\bf x}_{\xi{(\gamma})}={\bf b}_{\xi{(\gamma})}$, where $M$ and $N$ are constant, symmetric and
%positive
%definite matrices with the same sparsity pattern for a given level of discretization, and $\xi(\gamma)$ is a function on the
%stability parameter $\gamma$.
%Bertaccni \cite{Ber} presented a preconditioning technique for shifted linear
%systems based on updated incomplete factorization when $A$ is . When $A$ is
%Hermitian and the shifts are real, Freund \cite{Freund} considered
%some polynomial preconditioners for conjugate gradient type algorithms
%for the solution of large shifted linear systems. When $A$ is
%non-Hermitian, Jegerlehner proposed a polynomial
%preconditioner \cite{Je} and Gu {\it etal.} \cite{GZL} considered a flexible preconditioner for shifted linear systems.
However, when $A$ is non-Hermitian,
there are two major limitations to preconditioning the shifted linear systems, which
diminish the usefulness of precondition techniques considerably \cite{Je}. First, we have to
start with the same residual for all values of $\tau_i$, which means
that we cannot have $\tau_i$-dependent left preconditioning. Second,
preconditioning must retain the shifted structure of the matrix.
Indeed, as was pointed out
by Simoncini and Szyld \cite{Simoncini3}, standard preconditioning approaches may be effective
on each shifted system; however, they destroy the shifted structure so that the convenient
invariance property of the Krylov subspace can no longer be employed. In this
case, determining good preconditioners for shifted systems that preserve the original
structure is very important and is still an open area of research \cite{Simoncini3}.

In \cite{GZL}, Gu {\it et al.} presented a flexible preconditioner for solving the shifted
systems. This method allows one to incorporate different
preconditioners $(A-\tau_i I)^{-1}$ with different $\tau_i$, into the
Arnoldi procedure when constructing a projective subspace.
Unfortunately, one has to solve $m$ (complex) shifted systems $(A-\tau_i
I){\bf w}_i={\bf v}_i~(i=1,2\ldots,m)$ in each outer iteration of the
algorithm, where $m$ is the number of Arnoldi steps. Thus,
the expensive step of constructing the orthogonal basis can not be realized in real arithmetics even if $A$ and $B$ are real.
In particular, when direct methods are used, an additional difficulty is that for solving a sequence
of (complex) shifted systems, the matrix has to be refactorized for each shift, resulting in
considerable computational effort.
Just recently, this strategy was applied to precondition shifted linear systems arising in oscillatory hydraulic tomography (OHT) \cite{Sai}.
In \cite{ASG,Je}, some polynomial preconditioners are investigated to solve shifted linear systems together with the bi-conjugate gradient method (BiCG).
Unfortunately, for a general polynomial preconditioner, systems with different shifts may no
longer have equivalent Krylov subspaces for the shifted FOM (or GMRES) algorithm. Furthermore,
when the eigenvalues of $A$ are complex, the foci of the ellipse that encloses the spectrum has to be determined in advance \cite{ASG}.
Recently, a preconditioned and shifted (block) GMRES algorithm was proposed for accelerating the PageRank computation \cite{WWJ}. In this approach, a polynomial preconditioner
was introduced for shifted linear systems with multiple right-hand sides, under the
condition that the spectrum radius $\rho(A)<|\tau_i|,~i=1,2,\ldots,\nu$.
In many applications, however, we have that $\|A\|\gg |\tau_i|$ \cite{Dyn,Higham,LPS,P,TGB,Vorst}.
In this situation, the preconditioning strategy proposed in \cite{WWJ} does not work any more.

In this work,
we propose a new preconditioner for shifted linear systems with multiple right-hand sides when $\|A\|\gg |\tau_i|,~i=1,2,\ldots,\nu$, and present a preconditioned and shifted block FOM
algorithm with deflated restarting for the matrix exponential problem.
The motivation of our new method is threefold: First, we aim to precondition all the shifted linear system {\it simultaneously}. Second, we are interested in solving all the shifted
linear systems in the same search subspace. Third, when both $A$ and $B$ are real while the shifts are complex, the expensive step of constructing the orthogonal basis can be realized in real arithmetics.

This paper is organized as follows. In Section 2, we briefly introduce the Carath$\acute{\rm e}$ odory-Fej$\acute{\rm e}$r method for rational approximation
of exponential \cite{ST,Tre1,Tre2,Tre3}. In Section 3, we present a shifted block FOM algorithm with deflated restarting.
In Section 4, we propose a new preconditioner and provide a
preconditioned and shifted block FOM algorithm with deflation for the matrix exponential problem (\ref{3}).
Theoretical analysis is given to show the feasibility and rationality of the new approach. Numerical experiments
are reported in Section 5, illustrating that our new algorithm often outperforms many state-of-the-art algorithms for matrix exponential. Some
concluding remarks are given in Section 6.

In this paper, matrices are written as capitals such as $A$, and vectors are typeset in bold such as {\bf v}.
Let $\mathcal{K}_{m}(A,V_1)={\rm span}\{V_1,AV_1,\ldots,A^{m-1}V_1\}$ be the block Krylov subspace generated by $A$ and $V_1$, let $A^{\rm T}$ and $A^{\rm H}$ be
the transpose and conjugate transpose of $A$, respectively. Denote by $\|\cdot\|$ the vector norm or its induced matrix norm, and by $\|\cdot\|_F$ the F-norm of a matrix, respectively.
Let $I$ be the identity matrix and $O$ be a zero matrix or vector whose order
are clear from the context.
Let $E_j$ be an $mp\times p$ matix
which consists of the $\big((j-1)p+1\big)$-th to the $jp$-th ($1\leq j\leq m$) columns of the $mp\times mp$ identity matrix.
Standard MATLAB \cite{Matlab} notations are used whenever necessary.

\section{The Carath$\acute{\bf e}$odory-Fej$\acute{\bf e}$r method for rational approximation}
\setcounter{equation}{0}

A very powerful tool to manage matrix functions is to resort to a suitable rational approximation \cite{Higham,LS,Tre1,Tre2,Tre3,WW}.
Given a function $f(x)$ which is continuous on a closed interval, let $\mathcal{R}_{\mu,\nu}$ be the set of rational functions of the
form
\begin{equation}
\widehat{f}(x)=\frac{\sum_{k=0}^{\mu}a_k x^k}{\sum_{k=0}^{\nu}b_k x^k}.
\end{equation}
The well known Chebyshev approximation consists in finding an
approximation $f^{\star}(x)\in \mathcal{R}_{\mu,\nu}$, such that \cite{GP,Higham,Mein}
\begin{equation}
\|f(x)-f^{\star}(x)\|_{\infty}\leq \|f(x)-\widehat{f}(x)\|_{\infty},\quad\forall
\widehat{f}(x)\in\mathcal{R}_{\mu,\nu},
\end{equation}
where $\|f(x)-f^{\star}(x)\|_{\infty}$ stands for the infinity norm of $f(x)-f^{\star}(x)$. If $\mu=\nu$ and $f^{\star}(x)$ is the Chebyshev rational approximation, it holds
that \cite{CRV}
$$
\sup_{x\geq 0}\big|e^{-x}-f^{\star}(x)\big|\approx 10^{-\nu}.
$$
It is shown that $f^{\star}(x)$ always exists \cite{Mein}, however, its
uniqueness is not guaranteed unless $\nu=0$. Furthermore, how to
characterize the best approximation remains an issue \cite{LS,Mein}.

A practical alternative is to use the
Carath$\acute{\rm e}$odory-Fej$\acute{\rm e}$r (CF) approximation, and to work directly with rational functions designed as good approximations to $e^x$ on $\mathbb{R}^{-}$ (i.e., the left
half plane) \cite{ST,Tre1,Tre2,Tre3}.
This technique is based on the singular value decomposition of a Hankel
matrix of Chebyshev coefficients of the function
$e^x$,
transplanted from $\mathbb{R}^{-}$ to $[-1,1]$. The goal of this
method is to find a function $\widetilde{f}^{\star}(x)\in
\mathcal{\widetilde{R}}_{\mu,\nu}$ such that
\begin{equation}
\|f(x)-\widetilde{f}^{\star}(x)\|_{\infty}\leq
\|f(x)-\widetilde{f}(x)\|_{\infty},\quad\forall
\widetilde{f}(x)\in\mathcal{\widetilde{R}}_{\mu,\nu},
\end{equation}
where $\mathcal{\widetilde{R}}_{\mu,\nu}$ is the set of rational functions
of the following form \cite{GP,Tre1,Tre2,Tre3}
\begin{equation}
\widetilde{f}(x)=\frac{\sum_{k=-\infty}^{\mu}a_k x^k}{\sum_{k=0}^{\nu}b_k
x^k}.
\end{equation}
Then the Carath$\acute{\rm e}$odory-Fej$\acute{\rm e}$r approximation is obtained from dropping all terms of negative degree in the numerator of $\widetilde{f}^{\star}(x)$.

Although the Carath$\acute{\rm e}$odory-Fej$\acute{\rm e}$r approximation is in
principle only approximate, Magnus \cite{Mag} indicated that for
$\mu=\nu$, the CF approximation differs from the
best Chebyshev approximation for a factor of
$\mathcal{O}(56^{-\nu})$. Indeed, the error in exact arithmetic is below
standard machine precision for $\nu\geq 9$, and for $\nu=14$ it is
about $10^{-26}$ \cite{Mag,Tre2}.
Without loss of generality, we still denote by
$$
\widetilde{f}^{\star}(x)=\omega_0+\sum_{i=1}^\nu
\frac{\omega_i}{x-\tau_i}
$$
the Carath$\acute{\rm e}$odory-Fej$\acute{\rm e}$r approximation for approximating $e^x$ on
$\mathbb{R}^{-}$, where $\tau_i$ and $\omega_i~(i=1,2,\ldots,\nu)$ are the
poles and the residues, respectively. We refer to \cite{Tre2} for a MATLAB function to compute the
poles and the residues of the
Carath$\acute{\rm e}$odory-Fej$\acute{\rm e}$r method for the type $(\nu,
\nu)$ near best approximation.
The idea behind the Carath$\acute{\rm e}$odory-Fej$\acute{\rm e}$r approach for matrix exponential is to approximate $e^{A}B$ by using
\begin{equation}\label{11}
\widetilde{f}^{\star}(A)B=\omega_0B+\sum_{i=1}^\nu\omega_i(A-\tau_iI)^{-1}B.
\end{equation}
Thus, an approximation to $\widetilde{f}^{\star}(A)B$ can be obtained by first solving the shifted
linear systems appearing in the sum, and then by collecting the sum terms. When the shifts $\{\tau_i\}_{i=1}^\nu$ are appeared in conjugate pairs, one only needs to solve $\nu/2$ shifted linear system altogether.
\begin{rem}
In the Carath$\acute{e}$odory-Fej$\acute{e}$r approximation of the matrix exponential, we point out that the absolute
values of the shifts $\tau_i~(i=1,2,\ldots,\nu)$ are medium sized.
For instance, if we
choose $\nu=14$ in the Carath$\acute{e}$odory-Fej$\acute{e}$r approximation, then $\max_{1\leq i\leq\nu}|\tau_i|\approx18.9$.
In many applications, we have that $\|A\|\gg|\tau_i|~(i=1,2,\ldots,\nu)$ \cite{Dyn,Higham,LPS,P,TGB,Vorst}.
%We focus on this case in the paper.
\end{rem}

%If a Toeplitz matrix-vector product with respect to $T$ is wanted, we can first embed the Toeplitz matrix into
%a $2n\times2n$ circulant matrix, and the product is carried out with $\mathcal{O}(n log n)$ complexity \cite{CJ,Ng}. Nevertheless, in terms of (\ref{eqn3.1}), %four Toeplitz matrix-vector products require about
%sixteen FFTs of size $n$ to compute the matrix-vector product with respect to $T^{-1}$. To relieve the burden, one can factor the Toeplitz matrix in %$(\ref{eqn3.1})$ into two circulant matrices and two skew-circulant matrices \cite{LP,NSJ,Pang}:
%$$
%T^{-1}=\frac{1}{2\xi_{0}}\big[(L_{z}+R_{z}^{0})(R_{y}-L_{y}^{0})+(R_{y}+L_{y}^{0})(L_{z}-R_{z}^{0})\big],
%$$
%and the Toeplitz matrix-vector product with respect to $T^{-1}$ can be realized in six FFTs of length $n$. See \cite{LP,Pang} for more details.

\section{A shifted block FOM
algorithm with deflated restarting for shifted linear systems with multiple right-hand sides}

\setcounter{equation}{0}

A popular way to solve linear system with multiple
right-hand sides is the block approach
\cite{Saad,Simoncini3}.
In this section, we introduce a shifted block FOM
algorithm with the deflated restarting strategy \cite{Morgan}.
It can be viewed as a block and delated
version of the shifted FOM algorithm given by Simoncini
\cite{Simoncini}.

Suppose that $V_1$ is an $n\times p$ orthonormal matrix, then the block Arnoldi process with A. Ruhe's variant \cite{Saad} will
generate an orthnormal basis for the block Krylov subspace
$\mathcal{K}_{m+1}(A,V_1)=\mbox{span}\{V_1,AV_1,\ldots,A^{m}V_1\}$ in
exact arithmetics. %The
%block Arnoldi process (with A. Ruhe's variant) is described as
%follows, for more details, refer to \cite{Saad}.
%
%\begin{algorithm}
%{\rm\cite{Saad} The block Arnoldi process (with A. Ruhe's variant)}\\
%%---------------------------------------------------------------------------------------------------\\
%{\bf 1.}~{\it Start: Given $m$, the steps of the block Arnoldi
%process, $p$, the block size, and an initial orthonormal block
%vector $V_{1}$ of size $n\times p$; \\
%{\bf 2.}~Iterate:
%\begin{flushleft}
%~~{\bf for}$~j=p,p+1,\ldots,m\times p-1$ \\
%~~~~~~$k=j-p+1 $;\\
%~~~~~~${\bf w}=A{\bf v}_{k}$;\\
%~~~~~~{\bf for}~$i=1,2,\ldots,j$ \\
%~~~~~~~~~~$h_{i,k}=({\bf w},{\bf v}_{i})$;\\
%~~~~~~~~~~${\bf w}={\bf w}-h_{i,k}{\bf v}_{i}$;\\
%~~~~~~{\bf end}\\
%~~~~~ $h_{j+1,k}=\|{\bf w}\|_{2}; ~if~h_{j+1,k}=0, then~Stop$;\\
%~~~~~ ${\bf v}_{j+1}={\bf w}/h_{j+1,k}$;\\
%~~{\bf end}
%%---------------------------------------------------------------------------------------------------\\
%\end{flushleft}
%}
%\end{algorithm}
The following relation holds for the block Arnoldi process \cite{Saad}
\begin{eqnarray}\label{2.1}
A\mathcal{V}_m=\mathcal{V}_{m}\mathcal{H}_{m}+V_{m+1}H_{m+1,m}E^{\rm T}_{m}=\mathcal{V}_{m+1}\widetilde{\mathcal{H}}_m,
\end{eqnarray}
where $\mathcal{V}_{m+1}=\big[V_1,V_2,\ldots,V_{m+1}\big]$ is an orthogonal
basis for $\mathcal{K}_{m+1}(A,V_1)$,
$E_m$ is an $mp\times p$ matix
which consists of the last $p$ columns of the $mp\times mp$ identity
matrix, and $\mathcal{H}_m,\widetilde{\mathcal{H}}_m$ are $mp\times
mp$ and $(m+1)p\times mp$ upper band Hessenberg matrices,
respectively, with the $p\times p$ matrices $H_{i,j}$ being their
elements.

It is known that the block Krylov subspace is
spanned by the same basis if the matrix is scaled or shifted \cite{Saad}
$$
\mbox{span}\{V_1,(A-\tau_i I)V_1,\ldots,(A-\tau_i
I)^{m}V_1\}=\mbox{span}\{V_1,AV_1,\ldots,A^{m}V_1\},~ i=1,2,\ldots,\nu.
$$
%\begin{theorem}
By (\ref{2.1}), the block Arnoldi relations with respect to the
shifted matrices $A-\tau_i I$ are
\begin{eqnarray}\label{233}
(A-\tau_i I)\mathcal{V}_{m}&=&
\mathcal{V}_{m}(\mathcal{H}_{m}-\tau_i I)+V_{m+1}H_{m+1,m}E^{\rm T}_{m}\nonumber\\
&=&
\mathcal
{V}_{m+1}(\widetilde{\mathcal{H}}_m-\tau_i\widetilde{I}_m),\quad i=1,2,\ldots,\nu,
\end{eqnarray}
where
%$\widetilde{\mathcal{H}}_{m}^{i}=\widetilde{\mathcal{H}}_m-\tau_i\widetilde{I}_m$, and
$
\widetilde{I}_{m}
%=\left[\begin{array}{c} I\\
%O
%\end{array}\right]
$
is an $(mp+p)\times mp$ matrix which is the same as the identity matrix except for $p$ zero rows at the bottom.
However, the shifted block FOM algorithm will become
impractical when $m$ is
large, because of the growth of memory requirement and computation
cost.

%\end{theorem}
%\begin{proof}
%The proof is similar to that of Theorem 4.1 in \cite{WWJ}.
%It is easy to see from (\ref{2.1}) that
%\begin{eqnarray*}
%(A-\tau_i I)\mathcal{V}_{m}&=&A\mathcal{V}_{m}-\tau_{i}\mathcal{V}_{m}\\
%&=&\mathcal{V}_{m+1}\widetilde{\mathcal{H}}_{m}-\tau_i\mathcal
%{V}_{m+1}\widetilde{I}_m\\
%%&=&\Big(1-\frac{\alpha_{i}}{\alpha_{1}}\Big)V_{m}+\frac{\alpha_{i}}{\alpha_{1}}V_{m+1}
%%\overline{H}_{m}^{1}\\
%&=&\mathcal {V}_{m+1}\widetilde{\mathcal{H}}_{m}^{i},
%\end{eqnarray*}
%where
%$\widetilde{\mathcal{H}}_{m}^{i}=\widetilde{\mathcal{H}}_m-\tau_i\widetilde{I}_m,~i=1,2,\ldots,\nu$.
%\end{proof}
%However, when $A$ and $B$ are real, while the shifts $\{{\tau_i}\}'s$ are complex, the
%expensive steps for constructing orthogonal bases have to be performed in
%complex arithmetics.
%In this section, we present a shifted block FOM
%algorithm with deflation for linear systems with multiple right-hand sides and multiple
%complex shifts.
%We are ready to introduce a shifted block FOM algorithm for (\ref{3}).
%Rewrite (\ref{233}) as
%\begin{equation}\label{3.1}
%(A-\tau_i I)\mathcal{V}_{m}=
%\mathcal{V}_{m}(\mathcal{H}_{m}-\tau_i
%I)+V_{m+1}H_{m+1,m}E^{T}_{m},
%\quad i=1,2,\ldots,\nu.
%\end{equation}

One remedy is to
restart the shifted block FOM algorithm.
Denote by $\widehat{X}^{0}_{i}$ the approximations
obtained from the ``previous" cycle (where we choose $\widehat{X}^{0}_{i}=O$ for the first cycle), and by
$\widehat{R}_i^0=B-(A-\tau_i I)\widehat{X}^{0}_{i}$ the corresponding
residuals. Then the shifted block FOM algorithm uses
$\widehat{X}_i^{FOM}=\widehat{X}^{0}_{i}+\mathcal{V}_m\widehat{Z}_i^{FOM}~(i=1,2,\ldots,\nu)$ as
approximate solutions to (\ref{3}) in the ``current" cycle, with the residuals
\begin{equation}\label{3.2}
\widehat{R}_i^{FOM}=B-(A-\tau_i I)\widehat{X}^{FOM}_{i}\perp
\mbox{span}\{\mathcal{V}_m\},\quad i=1,2,\ldots,\nu.
\end{equation}
Let $\widehat{R}_i^0=V_1\widehat{R}_i$ be the QR factorizations of $\widehat{R}_i^0$, then it follows
from (\ref{3.2}) that $\widehat{Z}_i^{FOM}$ can be obtained from solving the following {\it projected}
linear systems
\begin{equation}\label{3.3}
(\mathcal{H}_{m}-\tau_i
I)\widehat{Z}_i^{FOM}=E_1\widehat{R}_i,~~i=1,2,\ldots,\nu,
\end{equation}
where $E_1$ is the matrix composed of the first $p$ columns of the $mp\times mp$ identity matrix.
Note that the
residuals with respect to different shifts are included in $\mbox{span}\{V_{m+1}\}$ in the shifted block FOM algorithm.
Indeed,
we have from (\ref{233}) that
\begin{eqnarray}\label{35}
%\begin{array}{lll}
\widehat{R}^{FOM}_{i}&=&B-(A-\tau_{i}I)\widehat{X}^{FOM}_{i}=\widehat{R}^{0}_{i}
-(A-\tau_{i}I)\mathcal{V}_{m}\widehat{Z}^{FOM}_{i}\nonumber\\%\nonumber
&=&\mathcal{V}_{m}\big[E_1\widehat{R}_i-(\mathcal{H}_{m}-\tau_i
I)\widehat{Z}^{FOM}_{i}\big]-V_{m+1}(H_{m+1,m}E^{\rm T}_{m}\widehat{Z}^{FOM}_{i})\nonumber\\%\nonumber
&=&-V_{m+1}(H_{m+1,m}E^{\rm T}_{m}\widehat{Z}^{FOM}_{i}),\quad i=1,2,\ldots,\nu,
%\end{array}
\end{eqnarray}
and
\begin{equation}\label{365}
\|\widehat{R}^{FOM}_{i}\|_F=\|H_{m+1,m}E^{\rm T}_{m}\widehat{Z}^{FOM}_{i}\|_F,\quad  i=1,2,\ldots,\nu.
\end{equation}
As a result, the residuals with respect to different shifts
are collinear in the shifted block FOM algorithm.
%From (\ref{3.3}),
%we obtain
%$
%\widehat{R}_i^{FOM}=V_{m+1}\widehat{S}_i
%$
%and $\|\widehat{R}_i^{FOM}\|_F=\|\widehat{S}_i\|_F$,
%where $\widehat{S}_i=
%-H_{m+1,m}E^{T}_{m}\widehat{Z}^{FOM}_{i},~i=1,2,\ldots,\nu$.

%As a result, this
%algorithm can be restarted by using $V_{m+1}$ as
%the initial block vector.
%The following
%theorem indicates that the residuals with respect to different shifts
%are collinear in the shifted block FOM algorithm.

%\begin{theorem}
%Let $\widehat{X}_i^{FOM}$ be the approximations obtained from the shifted
%block FOM algorithm, and $\widehat{R}_i^{FOM}$ be the corresponding block
%residuals. Then there are $p\times p$ matrices $\widehat{S}_i$ such
%that
%\begin{equation}\label{344}
%\widehat{R}_i^{FOM}=V_{m+1}\widehat{S}_i~~{\rm and}~~\|\widehat{R}_i^{FOM}\|_F
%=\|\widehat{S}_i\|_F,\quad i=1,2,\ldots,\nu.
%\end{equation}
%%and
%%\begin{equation}\label{3.5}
%%,~~i=1,2,\ldots,\nu.
%%\end{equation}
%\end{theorem}

Now we consider how to restart the shifted block FOM algorithm in practice. In this paper, we are interested in the
deflation strategy \cite{Morgan} for linear systems, in which the approximate eigenvectors are put
firstly in the search subspace. Here ``deflation" means computing eigenvectors corresponding
to some eigenvalues, and using them to
remove these eigenvalues from the spectrum of the matrix, to speed up the convergence of the iterative algorithm. More precisely, let $k$ be a multiple of $p$ and let $\widetilde{\bf y}_1,\widetilde{\bf y}_2,\ldots,\widetilde{\bf y}_k$ be $k$ Ritz vectors \cite{GV}
computed from the ``previous" cycle. Then after restarting, the search subspace of the ``current" cycle is \cite{Morgan}
\begin{equation}\label{3.6}
\mbox{span}\{\widetilde{\bf y}_1,\widetilde{\bf y}_2,\ldots,\widetilde{\bf y}_k,V_{m+1},AV_{m+1},\ldots,A^{q}V_{m+1}\},
\end{equation}
where $q=(mp-k)/p$. The shifted block FOM algorithm with deflation generates a block Arnoldi
relation similar to (\ref{2.1}), where
$\mathcal{V}_{m}$ is an $n\times mp$ matrix whose columns span the
subspace (\ref{3.6}), and $\widetilde{\mathcal{H}}_{m}$ is an $(mp+p)\times mp$ matrix that is band upper-Hessenberg except for a full $(k+p)\times k$ leading portion $\widetilde{\mathcal{H}}_k^{new}$. A part of
this recurrence can be separated out to give
$A\mathcal{V}_k^{new}=\mathcal{V}_{k+p}^{new}\widetilde{\mathcal{H}}_k^{new}$,
where $\mathcal{V}_{k+p}^{new}$ is an $n\times (k+p)$ matrix whose
columns span the subspace of Ritz vectors and $V_{m+1}$, and
$\mathcal{V}_k^{new}$ consists of the first $k$ columns of
$\mathcal{V}_{k+p}^{new}$; for more details, refer to \cite{JW,Morgan}.
Applying this deflation strategy to the shifted block FOM algorithm introduced above, we can present the following algorithm.
%shifted block FOM algorithm with
%deflated restarting for shifted linear systems with multiple right-hand sides.

\begin{algorithm}
{\rm A shifted block FOM algorithm with deflation for shifted linear systems with multiple right-hand sides}~~{\bf (SBFOM-DR)}\\
%---------------------------------------------------------------------------------------------------------------------------\\
{\bf 1.} Input: Given the block Arnoldi steps $m$, a prescribed
tolerance tol, and $k$ {\rm(}which is a multiple of $p${\rm)}, the number of approximate eigenvectors
retained from the previous cycle, and set all the initial guess
$\widehat{X}^{0}_{i}=O,~i=1,2,\ldots,\nu$; \\
{\bf 2.}~~Compute the QR decomposition: $\widehat{R}^{0}_{1}=V_{1}\widehat{R}_1$;\\
{\bf 3.}~~Use $V_{1}$ as the initial block vector, and run the block Arnoldi process {\rm(}with A. Ruhe's variant{\rm)} for the computation of $\mathcal{V}_{m}$ and
$\widetilde{\mathcal{H}}_{m}$;\\
{\bf 4.}~~Solve the projected shifted linear systems.
%\begin{flushleft}
%~~{\bf for} $i=1,2,\ldots,\nu$\\
%~~~~~~~~Solve the shifted systems (\ref{3.3});\\
%~~{\bf end}\\
%\end{flushleft}
If all the residual norms are below the prescribed tolerance tol, see {\rm(}\ref{365}{\rm)}, then Stop, else Continue;\\
{\bf 5.}~~Compute all the eigenpairs
$(\widehat{\theta}_{i},{\bf y}_{i})~(i=1,2,\ldots,mp)$ of
$\mathcal{H}_{m}$, and select $k$ smallest of them as
the desired ones;\\
{\bf 6.}~~Orthonormalization of the first $k$ short vectors:
Orthonormalize the $\{{\bf y}_{i}\}$'s, first separating them into
real parts and imaginary parts if they are complex, in order to form
an $mp\times k$ matrix
$P_{k}=[{\bf y}_{1},{\bf y}_{2},\ldots,{\bf y}_{k}]$. Both parts of
complex vectors need to be included, so $k$ may be increased or decreased if necessary;\\
{\bf 7.}~~Orthonormalization of the $k+p$ short vectors: Extend
$P_{k}$ to an $(mp+p)\times k$ matrix $\widehat{P}_{k}$, by appending a
$p\times k$ zero matrix at the bottom, and set $P_{k
+p}=[\widehat{P}_{k},~\widehat{E}]$, where $\widehat{E}$ is an $(mp+p)\times
p$ zeros matrix except for its last $p$ rows being an identity
matrix. Note that $P_{k +p}$ is an
$(mp+p)\times(k +p)$ matrix;\\
{\bf 8.}~~Form portions of new $\widetilde{\mathcal{H}}_{mp}$ and
$\mathcal{V}_{m}$ using the old $\widetilde{\mathcal{H}}_{mp}$ and
$\mathcal{V}_{m+1}$: Let $\widetilde{\mathcal{H}}^{new}_{k}=P^{\rm T}_{k
+p}\widetilde{\mathcal{H}}_{mp}P_{k}$, and $\mathcal{V}^{new}_{k
+p}=\mathcal{V}_{m+1}P_{k+p}$;\\
{\bf 9.}~~Block Arnoldi iteration: Apply the block Arnoldi iteration
from the current point to form the rest portions of
$\mathcal{V}_{m+1}$ and $\widetilde{\mathcal{H}}_{mp}$, where the
current point is $\mathcal{V}^{new}_{k +p}$, and goto {\bf Step 4}.
%---------------------------------------------------------------------------------------------------------------------------
\end{algorithm}

%\begin{rem}
Algorithm 1 is very attractive when both
$A$ and $B$ are real while the shifts $\{\tau_i\}$'s are complex. Indeed,
at each cycle after restarting, all the complex residuals are
collinear to the $(m+1)$-st real basis vector $V_{m+1}$, and the expensive step for constructing the orthogonal
basis $\mathcal{V}_{m+1}$ can be performed in real arithmetics.
%\end{rem}

\section{Preconditioning the restarted and shifted block FOM algorithm}
\setcounter{equation}{0}

In this section, we propose a new preconditioner for shifted linear systems with multiple right-hand sides when $\|A\|\gg |\tau_i|~(i=1,2,\ldots,\nu)$, and apply it
to the computation of the matrix exponential problem. Some theoretical results are established to show the rationality and feasibility of our new preconditioning strategy.

%However, it was stressed by Darnnu {\it et al.} that with a
%preconditioner,
%systems with different shifts would (generally) no
%longer have equivalent Krylov subspaces for the shifted GMRES algorithm \cite{DMW}. In other words,
%it is a hard task to apply the preconditioning techniques to the
%restarted and shifted (block) FOM or GMRES algorithm.
%In this section, we consider how to precondition the shifted linear systems with multiple righthand sides efficiently.
%Some theoretical results are given to show rationality of our strategy.

\subsection{The preconditioner and the algorithm}
In many applications, we have that $\|A\|\gg \max_{1\leq i\leq\nu}|\tau_i|$ \cite{Dyn,Higham,LPS,P,TGB,Vorst}.
For instance, if $A$ is the usual finite difference/element discretization of a Poisson operator with Dirichlet boundary conditions,
the norm of $A$ grows rapidly with respect to the number $n$ of spatial grid points \cite{QV}.
Indeed, if $\|A\|>|\tau_i|$, then
$$
\frac{\|(A-\tau_{i}I)-A\|}{\|A-\tau_{i}I\|}\leq\frac{|\tau_{i}|}{\|A\|-|\tau_{i}|},\quad i=1,2,\ldots,\nu.
$$
Specifically, if $\|A\|\gg|\tau_i|$, we have that
\begin{equation}\label{43}
\frac{\|(A-\tau_{i}I)-A\|}{\|A-\tau_{i}I\|}\lessapprox\frac{|\tau_{i}|}{\|A\|}\ll 1,\quad i=1,2,\ldots,\nu.
\end{equation}
Thus, the idea is to use
\begin{equation}
\widetilde{M}_{i}=A^{-1},\quad i=1,2,\ldots,\nu,
\end{equation}
as right-preconditioner to (\ref{3}). In other words, the matrix $A$ can be viewed as an approximation to $A-\tau_iI$ if $\|A\|\gg |\tau_i|~(i=1,2,\ldots,\nu)$,
and our strategy can be understood as an approximate inverse preconditioning or a shift-invert technique where the shift is set to be zero.
Recall from Remark 2.1 that the assumption of $\|A\|\gg \max_{1\leq i\leq\nu}|\tau_i|$ is mild for the
Carath$\acute{\rm e}$odory-Fej$\acute{\rm e}$r approximation of the matrix exponential.
Moreover, it is seen that
\begin{equation}
\|I-(A-\tau_i I)A^{-1}\|=\|\tau_iA^{-1}\|=\frac{|\tau_i|}{\|A\|}\cdot\kappa(A),\quad i=1,2,\ldots,\nu,
\end{equation}
%moreover, if $\|\tau_iA^{-1}\|\ll 1$, we obtain
%\begin{equation}
%\frac{\|(A-\tau_iI)^{-1}-A^{-1}\|}{\|A^{-1}\|}\lessapprox\|\tau_iA^{-1}\|=\frac{|\tau_i|}{\|A\|}\cdot\kappa(A).
%\end{equation}
where $\kappa(A)=\|A\|\cdot\|A^{-1}\|$ is the condition number of $A$. Therefore,
$A^{-1}$ will be a good preconditioner to the shifted linear systems
if $\max_{1\leq i\leq\nu}\frac{|\tau_i|}{\|A\|}$
is sufficiently small and $A$ is not too ill-conditioned.
Furthermore, the larger $\|A\|$ is, the better the preconditioner will be.
Indeed, using $A^{-1}$ as the preconditioner may
be advantageous with respect to $(A-\tau_i I)^{-1}$ if either of the following
conditions apply {\rm\cite{Simoncini4}}:
First, solving with $A-\tau_i I$ is more expensive or difficult than solving with $A$ {\rm(}e.g., $A$ is large
and real while the $\{\tau_i\}$'s are complex{\rm)}.
Second, the number of shifts involved is large, and the dimension of the search space
is significantly low. Both conditions can be satisfied in many real applications {\rm\cite{DK2,Simoncini4}}.
%\begin{rem}
%When $A$ is singular, we can choose $\eta\in\mathbb{R}$, such that $A-\eta I$ is nonsingular and
%$
%\frac{\|(A-\tau_{i}I)-(A-\eta I)\|}{\|A-\tau_{i}I\|}~(i=1,2,\ldots,\nu)
%$
%is as small as possible, and use $\widetilde{M}_i=(A-\eta I)^{-1}~(i=1,2,\ldots,\nu)$ as the preconditioners for the shifted linear systems.
%\end{rem}

Therefore, the preconditioned linear systems turn out to be the following form
$$
(A-\tau_iI)\widetilde{M}_{i}{Y}_{i}=B,\quad i=1,2,\ldots,\nu,
$$
that is,
\begin{equation}\label{3188}
\big(I-\tau_{i}
A^{-1}\big){Y}_{i}=B,\quad i=1,2,\ldots,\nu,
\end{equation}
with
${X}_{i}=\widetilde{M}_{i}{Y}_{i}=A^{-1}{Y}_{i}~(i=1,2,\ldots,\nu)$
being the desired solutions.

\begin{rem}
It is known that clustered eigenvalues are favorable for convergence of Krylov subspace methods \cite{Saad}.
Let $\lambda_j$ be eigenvalues of $A$, then the eigenvalues of $(I-\tau_{i}A^{-1})$ are mapped to be $1-\tau_i/\lambda_j$. If the smallest eigenvalues of $A$ are deflated using the technique described in Section 3, then the eigenvalues of $(I-\tau_{i}A^{-1})$ will be clustered around 1 as $|\lambda_j|\gg|\tau_i|$. This gives the rationality of knitting the deflated restarting technique together with our preconditioning strategy.
\end{rem}

%The following theorem indicates that the preconditioned linear
%systems (\ref{3188}) with different shifts will have equivalent
%Krylov subspaces.

%\begin{theorem}
We notice that the block Krylov subspace with
respect to $A^{-1}$ and those
with respect to $(A-\tau_{i}I)\widetilde{M}_i
~(i=1,2,\ldots,\nu)$ are identical. Moreover,
suppose that the
block Arnoldi relation with respect to $A^{-1}$ is
\begin{eqnarray}\label{4.5}
A^{-1}\mathcal{V}_{m}=\mathcal{V}_{m}\mathcal{H}_{m}+V_{m+1}H_{m+1,m}E^{\rm T}_{m}=\mathcal{V}_{m+1}\widetilde{\mathcal{H}}_m.
\end{eqnarray}
Recall that both $\mathcal{V}_{m+1}$ and $\widetilde{\mathcal{H}}_m$ are different from those in (\ref{2.1}).
Then, the block Arnoldi relations with respect to the shift systems (\ref{3188})
are
\begin{equation}\label{4.6}
\big(I-\tau_{i}
A^{-1}\big)\mathcal{V}_{m}=\mathcal{V}_{m+1}\big(\widetilde{I}_{m}-\tau_{i}\widetilde{\mathcal{H}}_m\big),\quad i=1,2,\ldots,\nu.
\end{equation}
So we can precondition
all the shifted linear systems simultaneously, and solve them
in the same search subspace in the preconditioned and shifted block FOM algorithm.

Denote by $\widehat{Y}^{0}_{i}$ the approximations obtained from the
``previous" cycle of the preconditioned and shifted block FOM algorithm, and by
$\widehat{R}_i^0=B-(I-\tau_iA^{-1})\widehat{Y}^{0}_{i}~(i=1,2,\ldots,\nu)$ the corresponding
residuals. Let $\widehat{R}_i^0=V_1\widehat{R}_i$ be the QR
factorization,
then the preconditioned and shifted block FOM algorithm uses
$\widehat{Y}_i^{FOM}=\widehat{Y}^{0}_{i}+\mathcal{V}_m\widehat{Z}_i^{FOM}$ as
approximate solutions to (\ref{3188}), where $\widehat{Z}_i^{FOM}$ are the solutions of the following {\it projected} linear systems
\begin{equation}\label{3236}
(I-\tau_i\mathcal{H}_{m})\widehat{Z}_i^{FOM}=E_1\widehat{R}_i,\quad i=1,2,\ldots,\nu.
\end{equation}
The residuals are
\begin{eqnarray}
%\begin{array}{lll}
\widehat{R}^{FOM}_{i}&=&B-(I-\tau_{i}A^{-1})(\widehat{Y}^{0}_{i}+\mathcal{V}_m\widehat{Z}_i^{FOM})\nonumber\\
&=&\tau_iV_{m+1}(H_{m+1,m}E^{\rm T}_{m}\widehat{Z}^{FOM}_{i}),\quad i=1,2,\ldots,\nu,
%\end{array}
\end{eqnarray}
and
\begin{equation}\label{49}
\|\widehat{R}^{FOM}_{i}\|_F=|\tau_i|\cdot\|H_{m+1,m}E^{\rm T}_{m}\widehat{Z}^{FOM}_{i}\|_F,\quad i=1,2,\ldots,\nu.
\end{equation}

We are ready to propose the main algorithm of this paper for shifted linear systems
with multiple right-hand sides as $\|A\|\gg |\tau_i|,~i=1,2,\ldots,\nu$.

\begin{algorithm}
{\rm ~A preconditioned and shifted block FOM algorithm with deflation for shifted linear systems with multiple right-hand sides}~{\bf
(PSBFOM-DR)}\\
Steps {\bf 1}--{\bf 2} are the same as those in Algorithm 1;\\
{\bf 3.}~~Run the block Arnoldi process {\rm(}with A. Ruhe's variant{\rm)} for the computation of {\rm(}\ref{4.6}{\rm)};\\
{\bf 4.}~~Solving the projected shifted linear systems.
%\begin{flushleft}
%~{\bf for} $i=1,2,\ldots,\nu$\\
%~~~~~~Solving the projected systems (\ref{3236});\\
%~{\bf end}\\
%\end{flushleft}
If the residual norms are below the prescribed tolerance tol, see {\rm(\ref{49}\rm}), then Stop and form
the numerical solutions $\widetilde{X}_i~(i=1,2,\ldots,\nu)$,                                                                                                                  else continue;\\
{\bf 5.}~~Steps {\bf 5}--{\bf 9} are similar to those in Algorithm 1, except for the block Krylov subspace is generated by using $A^{-1}$.
\end{algorithm}

In each cycle, some linear
systems with respect to $A$ must be solved in the block Arnoldi process to expand the basis. More precisely, $mp$ for the first cycle and $mp-k$ for the cycles after the first. This
makes the
generation of the search space more expensive than the unpreconditioned one.
In practical calculations, if $A$ is of medium sized (say, $n\leq 10,000$), we can perform the sparse LU factorization of $A$ for the inverse.
Moreover,
the LU factorization requires to be performed once and for all, and the $L$ and $U$ factors can be stored for later use.
Specifically, when $A$ has some special structure such as
Toeplitz, one can use the Gohberg--Semencul formula \cite{GSF} for the matrix-vector products with respect to $A^{-1}$, with no need to form and store the inverse explicitly; see
\cite{LPS,PS} for more details.

On the other hand,
if the matrix is so
large that direct methods are prohibitive, using an iterative solver is definitely possible and advisable \cite{Saad,Simoncini3}.
%one can iteratively solve the linear systems with respect to $A$,
%say, by using some preconditioned and restarted
%(block) Krylov subspace methods.
Another alternative is to use some inexact Krylov subspace algorithms or the flexible preconditioning strategy in which the products with respect to $A^{-1}$ can be computed inaccurately as the outer
iteration
converges \cite{BF1,Simoncini2}.
In \cite{Sai}, an inexact preconditioning was considered for solving the shifted linear systems. One can naturally extend the idea of inexact preconditioning to
systems with multiple shifts and multiple right-hand sides using block and deflation techniques. We shall not pursue this issue here.

\subsection{Theoretical analysis}

In this subsection, we first establish a relationship between the approximations obtained from the shifted block FOM
method and the shifted block GMRES method.
We then show the importance of the convergence of Ritz pairs to the new algorithm.
Finally, we give an error analysis on the approximate solution obtained from solving (\ref{11}) when the inverse of $A$
is computed inaccurately.

As the shifted block GMRES algorithm and the shifted block FOM algorithm are two popular approaches for shifted linear systems with multiple
right-hand sides, we are interested in the relation between the approximations of these two methods.
Denote by $\widehat{Y}^{0}_{i}$ the
initial guesses, and by
$\widehat{R}_i^0=B-(I-\tau_iA^{-1})\widehat{Y}^{0}_{i}$ the
residuals, $i=1,2,\ldots,\nu$. Let $\widehat{R}_i^0=V_1\widehat{R}_i$ be the QR
factorization of $\widehat{R}_i^0$.
Then the preconditioned and shifted block FOM algorithm uses
$\widehat{Y}_i^{FOM}=\widehat{Y}^{0}_{i}+\mathcal{V}_m\widehat{Z}_i^{FOM}$ as
approximate solutions to (\ref{3188}), where
\begin{equation}\label{323}
\widehat{Z}_i^{FOM}=(I-\tau_i\mathcal{H}_{m})^{-1}E_1\widehat{R}_i,\quad i=1,2,\ldots,\nu.
\end{equation}
Next we consider the solutions derived from the preconditoned and shifted block GMRES algorithm \cite{WWJ}.
Without loss of generality, let $\big(I-\tau_{1}
A^{-1}\big){Y}_{1}=B$ be the ``seed" block linear system, and let $\big(I-\tau_{i}
A^{-1}\big){Y}_{i}=B~(i=2,3,\ldots,\nu)$ be the ``additional" block linear systems \cite{WWJ}. In the preconditioned and shifted block GMRES algorithm, only the ``seed" block system residual is minimized, whereas
the residuals of the
``additional" block systems are forced to be collinear with that of the ``seed" block system
after restarting \cite{WWJ}. More precisely, for the ``seed" block system, the preconditioned and shifted block GMRES algorithm uses
$\widehat{Y}_1^{GMRES}=\widehat{Y}^{0}_{1}+\mathcal{V}_m\widehat{Z}_1^{GMRES}$ as an
approximate solution to the ``seed" system, where
$$
\widehat{Z}_1^{GMRES}={\mbox {\rm arg}\min}_{Z\in\mathbb{C}^{mp\times p}}\|E_1\widehat{R}_1-\big(\widetilde{I}_{m}-\tau_{1}\widetilde{\mathcal{H}}_m\big)Z\|_F,
$$
or equivalently,
$$
\big(\widetilde{I}_{m}-\tau_{1}\widetilde{\mathcal{H}}_m\big)^{\rm H}\big(\widetilde{I}_{m}-\tau_{1}\widetilde{\mathcal{H}}_m\big)\widehat{Z}_1^{GMRES}
=\big(\widetilde{I}_{m}-\tau_{1}\widetilde{\mathcal{H}}_m\big)^{\rm H}E_1\widehat{R}_1.
$$
If $I-\tau_{1}\mathcal{H}_m$ is nonsingular, the above equation can be reformulated as
\begin{equation}\label{eqn410}
\big[(I-\tau_{1}\mathcal{H}_m)+|\tau_1|^2(I-\tau_{1}\mathcal{H}_m)^{\rm -H}E_mH_{m+1,m}^{\rm T}H_{m+1,m}E_m^{\rm T}\big]\widehat{Z}_1^{GMRES}=E_1\widehat{R}_1.
\end{equation}
Denote $\Gamma_1=|\tau_1|^2(I-\tau_{1}\mathcal{H}_m)^{\rm -H}E_m$, $\Gamma_2^{\rm T}=H_{m+1,m}^{\rm T}H_{m+1,m}E_m^{\rm T}$, and
\begin{equation}
\Omega_1=(I-\tau_{1}\mathcal{H}_m)^{-1}
\Gamma_1\big(I+\Gamma_2^{\rm T}(I-\tau_{1}\mathcal{H}_m)^{-1}\Gamma_1\big)^{-1}\Gamma_2^{\rm T},
\end{equation}
by the Sherman-Morrison-Woodbury formula \cite{GV}, we have that
\begin{eqnarray}\label{eq411}
\widehat{Z}_1^{GMRES}&=&(I-\tau_1\mathcal{H}_{m})^{-1}E_1\widehat{R}_1-\Omega_1(I-\tau_1\mathcal{H}_{m})^{-1}E_1\widehat{R}_1\nonumber\\
&=&\widehat{Z}_1^{FOM}-\Omega_1\widehat{Z}_1^{FOM}=(I-\Omega_1)\widehat{Z}_1^{FOM}.
\end{eqnarray}

Next we consider the approximations of the ``additional" block systems computed by the preconditioned and shifted block GMRES algorithm. Let  $\widehat{Y}_i^{GMRES}=\widehat{Y}^{0}_{i}+\mathcal{V}_m\widehat{Z}_i^{GMRES}~(i=2,3,\ldots,\nu)$ be the approximate solutions, and let $\widehat{R}_1^{GMRES}$
be the residual of the ``seed" system. Then the residuals of the ``additional" systems can be expressed as \cite{WWJ}
\begin{eqnarray}
\widehat{R}^{GMRES}_{i}=\widehat{R}^{GMRES}_{1}W_{i},\qquad i=2,3,\ldots, \nu,
\end{eqnarray}
where $W_{i}$ are some $p\times p$ nonsingular matrices. It follows that
$$
E_1\widehat{R}_i-(\widetilde{I}_{m}-\tau_{i}\widetilde{\mathcal{H}}_m)\widehat{Z}_i^{GMRES}=
\big[E_1\widehat{R}_1-(\widetilde{I}_{m}-\tau_{1}\widetilde{\mathcal{H}}_m)\widehat{Z}_1^{GMRES}\big]W_i,\quad i=2,3,\ldots, \nu.
$$
%Denote $\Psi=E_1\widehat{R}_i-(\widetilde{I}_{m}-\tau_{1}\widetilde{\mathcal{H}}_m)Z_1^{GMRES}$,
Denote $G_1^{GMRES}=E_1\widehat{R}_1-(\widetilde{I}_{m}-\tau_{1}\widetilde{\mathcal{H}}_m)\widehat{Z}_1^{GMRES}$, the above relation can be rewritten as
\begin{eqnarray*}
\left[\begin{array}{cc} \widetilde{I}_{m}-\tau_{i}\widetilde{\mathcal{H}}_m& G_1^{GMRES}
\end{array}\right]
\left[\begin{array}{c}
\widehat{Z}_i^{GMRES}\\W_{i}
\end{array}\right]
=E_{1}\widehat{R}_i,\qquad i=2,3,\ldots,\nu.
\end{eqnarray*}
Let $\Psi_1=G_1^{GMRES}(1:mp,:),\Psi_2=G_1^{GMRES}(mp+1:mp+p,:)$ be the $mp\times p$ and $p\times p$ matrices composed of the first $mp$ rows and the last
$p$ rows of $G_1^{GMRES}$, respectively. Denote by $\Phi_i=[\widetilde{I}_{m}-\tau_{i}\widetilde{\mathcal{H}}_m,~G_1^{GMRES}]^{-1}(1:mp,1:mp)$ the matrix composed of the first $mp$ rows and $mp$ columns of the inverse of $[\widetilde{I}_{m}-\tau_{i}\widetilde{\mathcal{H}}_m,~G_1^{GMRES}]$. If $\Psi_2$ is nonsingular, then
$$
\Phi_i=\big((I-\tau_{i}\mathcal{H}_m)+\tau_i\Psi_1\Psi_2^{-1}H_{m+1,m}E_m^{\rm T}\big)^{-1},\quad i=2,\ldots,\nu.
$$
If $I-\tau_{i}\mathcal{H}_m$ is nonsingular, denote $\Psi_3^{\rm T}=\Psi_2^{-1}H_{m+1,m}E_m^{\rm T}$ and
\begin{equation}
\Omega_i=\tau_i(I-\tau_{i}\mathcal{H}_m)^{-1}
\Psi_1\big(I+\tau_i\Psi_3^{\rm T}(I-\tau_{i}\mathcal{H}_m)^{-1}\Psi_1\big)^{-1}\Psi_3^{\rm T},\quad i=2,\ldots,\nu,
\end{equation}
from the Sherman-Morrison-Woodbury formula \cite{GV}, we obtain
\begin{eqnarray}\label{eq414}
\widehat{Z}_i^{GMRES}&=&\big((I-\tau_{i}\mathcal{H}_m)+\tau_i\Psi_1\Psi_2^{-1}H_{m+1,m}E_m^{\rm T}\big)^{-1}E_1\widehat{R}_i\nonumber\\
&=&\widehat{Z}_i^{FOM}-\Omega_i\widehat{Z}_i^{FOM}=(I-\Omega_i)\widehat{Z}_i^{FOM},~ i=2,3,\ldots,\nu.
\end{eqnarray}
In conclusion, we have the following theorem.
\begin{theorem}
Denote by $\widehat{Y}^{0}_{i}$ the initial guesses of the preconditioned and shifted block FOM and GMRES algorithms, and let
$\widehat{Y}_i^{FOM}=\widehat{Y}^{0}_{i}+\mathcal{V}_m\widehat{Z}_i^{FOM}$,
$\widehat{Y}_i^{GMRES}=\widehat{Y}^{0}_{i}+\mathcal{V}_m\widehat{Z}_i^{GMRES}~(i=1,2,\ldots,\nu)$ be the approximate solutions obtained from the two algorithms. Then under the above assumptions and notations, $\widehat{Z}_1^{FOM}$ and $\widehat{Z}_1^{GMRES}$ satisfy {\rm (\ref{eq411})}, and $\widehat{Z}_i^{FOM}$ and $\widehat{Z}_i^{GMRES}$ satisfy
{\rm(\ref{eq414})}, $i=2,\ldots,\nu$, respectively.
\end{theorem}

Next, we show importance of the accuracy of Ritz
pairs for the convergence of the preconditioned and shifted block FOM algorithm.
The following theorem establishes a relation between the residuals of the shifted block linear systems
and those of the Ritz pairs.
\begin{theorem}
Let
$\widehat{R}_i^{FOM}=\tau_iV_{m+1}H_{m+1,m}(E_m^{\rm T}\widehat{Z}_i^{FOM}),~i=1,2,\ldots,\nu$,
be the residuals of the preconditioned and shifted block FOM algorithm. Denote by $(\mu_j,\mathcal{V}_m{\bf y}_j)$ the Ritz pairs of $A^{-1}$, and by ${\bf r}_j=A^{-1}\mathcal{V}_m{\bf y}_j-\mu_j\mathcal{V}_m{\bf y}_j~(j=1,2,\ldots,mp)$ the residuals of the Ritz pairs.
If $\mathcal{H}_m=P\Lambda P^{-1}$ is the spectrum decomposition of $\mathcal{H}_m$, where $P=[{\bf y}_1,{\bf y}_2,\ldots,{\bf y}_{mp}]$ and $\Lambda={\mbox diag}(\mu_1,\mu_2,\ldots,\mu_{mp})$ is diagonal, then
\begin{equation}
\widehat{R}_i^{FOM}=\tau_i\big[{\bf r}_1,{\bf r}_2,\ldots,{\bf r}_{mp}\big]\big(P(I-\tau_i\Lambda)\big)^{-1}\big(E_1\widehat{R}_i\big),\quad i=1,2,\ldots,\nu.
\end{equation}
\end{theorem}
\begin{proof}
Recall from (\ref{4.5}) that the Ritz residuals can be rewritten as
\begin{eqnarray}\label{eq416}
{\bf r}_j&=&A^{-1}\mathcal{V}_m{\bf y}_j-\mu_j\mathcal{V}_m{\bf y}_j=V_{m+1}H_{m+1,m}(E_m^{\rm T}{\bf y}_i),\quad j=1,2,\ldots,mp.
\end{eqnarray}
Moreover, we have from (\ref{323}) that
\begin{eqnarray}\label{eq417}
\widehat{R}_i^{FOM}=\tau_iV_{m+1}H_{m+1,m}E_m^{\rm T}(I-\tau_i\mathcal{H}_{m})^{-1}E_1\widehat{R}_i.
\end{eqnarray}
If $\mathcal{H}_m=P\Lambda P^{-1}$ is diagonalizable, then
\begin{eqnarray}\label{eq418}
(I-\tau_i\mathcal{H}_{m})^{-1}E_1\widehat{R}_i&=&P(I-\tau_i\Lambda)^{-1}P^{-1}E_1\widehat{R}_i.
\end{eqnarray}
Thus, it follows from (\ref{eq416})--(\ref{eq418}) that
\begin{eqnarray*}
\widehat{R}_i^{FOM}&=&\tau_iV_{m+1}H_{m+1,m}E_m^{\rm T}P(I-\tau_i\Lambda)^{-1}P^{-1}E_1\widehat{R}_i\\
&=&\tau_iV_{m+1}H_{m+1,m}E_m^{\rm T}[{\bf y}_1,{\bf y}_2,\ldots,{\bf y}_{mp}](I-\tau_i\Lambda)^{-1}P^{-1}E_1\widehat{R}_i\\
&=&\tau_i\big[{\bf r}_1,{\bf r}_2,\ldots,{\bf r}_{mp}\big](I-\tau_i\Lambda)^{-1}P^{-1}E_1\widehat{R}_i,\quad i=1,2,\ldots,\nu.
\end{eqnarray*}
\end{proof}

Finally, we give an error analysis on the approximation (\ref{11}).
Let $X_i,\widetilde{X}_i$ be the ``exact" and the ``computed" solutions of the shifted linear systems (\ref{3}), respectively, and denote by
\begin{equation}
\widetilde{Z}=\omega_0B+\sum_{i=1}^\nu\widetilde{X}_i
\end{equation}
the numerical approximation to $e^AB$. Then we have from (\ref{11}) that
\begin{eqnarray}
\|e^AB-\widetilde{Z}\|&=&\|e^AB-\widetilde{f}^{\star}(A)+\widetilde{f}^{\star}(A)-\widetilde{Z}\|\nonumber\\
&\leq&\|e^AB-\widetilde{f}^{\star}(A)\|+\|\widetilde{f}^{\star}(A)-\widetilde{Z}\|\nonumber\\
&\leq&\|e^AB-\widetilde{f}^{\star}(A)\|+\sum_{i=1}^\nu|\omega_i|\cdot\|X_i-\widetilde{X}_i\|.
\end{eqnarray}
Thus, the error is dominated by the one from the Carath$\acute{\rm e}$odory-Fej$\acute{\rm e}$r approximation, and the one from the numerical solutions of the shifted linear systems. In this paper, we are interested in the latter.

Let $\widetilde{A}^{-1}$ be the ``computed" solution of the exact inverse $A^{-1}$. Then the preconditioned shifted linear systems (\ref{3188}) turn out to be
\begin{equation} \label{eq421}
(I-\tau_i\widetilde{A}^{-1}\big)\widetilde{Y}_{i}=B,\quad i=1,2,\ldots,\nu.
\end{equation}
%and the solutions to (\ref{3}) are $\widetilde{X}_i=\widetilde{A}^{-1}\widetilde{Y}_i,~i=1,2,\ldots,\nu$.
The following theorem shows that the quality of the approximations
to the solutions of (\ref{3}) is bounded by the relative accuracy of $\widetilde{A}^{-1}$ and the condition number of $A$.
\begin{theorem}
Let $Y_i=(I-\tau_i{A}^{-1})^{-1}B,~\widetilde{Y}_i=(I-\tau_i\widetilde{A}^{-1})^{-1}B$, and ${X}_i={A}^{-1}{Y}_i,~\widetilde{X}_i=\widetilde{A}^{-1}\widetilde{Y}_i$ be the exact and the computed solutions of {\rm(}\ref{3}{\rm)}, respectively. Denote $F=\widetilde{A}^{-1}-A^{-1}$, if $\|\tau_i(I-\tau_i A^{-1})^{-1}F\|\ll 1$, then
\begin{equation}\label{eqn421}
\frac{\|X_i-\widetilde{X}_i\|}{\|X_i\|}\lessapprox\kappa(A)\|(I-\tau_iA^{-1})^{-1}\|\cdot\frac{\|\widetilde{A}^{-1}-A^{-1}\|}{\|A^{-1}\|},\quad i=1,2,\ldots,\nu.
\end{equation}
%where $\kappa(A)=\|A\|\cdot\|A^{-1}\|$ is the condition number of $A$.
%Furthermore, if $\|\tau_iA^{-1}\|\ll 1$ and $\|\tau_i\widetilde{A}^{-1}\|\ll 1$, then
%\begin{equation}
%\frac{\|X_i-\widetilde{X}_i\|}{\|X_i\|}\leq\|A-\tau_iI\|\cdot\|F\|,\quad i=1,2,\ldots,\nu.
%\end{equation}
\end{theorem}
\begin{proof}
We have from (\ref{3188}) and (\ref{eq421}) that
\begin{equation}\label{eq422}
\widetilde{Y}_i-Y_i=\tau_i(\widetilde{A}^{-1}\widetilde{Y}_i-A^{-1}Y_i)=\tau_i(\widetilde{X}_i-X_i).
\end{equation}
Let $F=\widetilde{A}^{-1}-A^{-1}$, then
\begin{eqnarray*}
\widetilde{Y}_i-Y_i&=&(I-\tau_i\widetilde{A}^{-1})^{-1}B-(I-\tau_i{A}^{-1})^{-1}B\nonumber\\
&=&\big[(I-\tau_i\big(I-\tau_i{A}^{-1})^{-1}F\big)^{-1}-I\big](I-\tau_i{A}^{-1})^{-1}B\nonumber\\
&=&\big[(I-\tau_i\big(I-\tau_i{A}^{-1}\big)^{-1}F)^{-1}-I\big]Y_i.
\end{eqnarray*}
If $\|\tau_i(I-\tau_i A^{-1})^{-1}F\|\ll 1$, then we have that
\begin{eqnarray*}
\|(I-\tau_i(I-\tau_i{A}^{-1})^{-1}F)^{-1}-I\|&=&\|\tau_i(I-\tau_i{A}^{-1})^{-1}F\|+\mathcal{O}\big(\|\tau_i(I-\tau_i{A}^{-1})^{-1}F\|^2\big),
%&\approx&\|\tau_i(I-\tau_i{A}^{-1})^{-1}F\|,
\end{eqnarray*}
and
\begin{eqnarray}\label{eq423}
\frac{\|\widetilde{Y}_i-Y_i\|}{|\tau_i|}\lessapprox\|(I-\tau_i{A}^{-1})^{-1}F\|\cdot\|Y_i\|
\leq\|(I-\tau_i{A}^{-1})^{-1}F\|\cdot\|A\|~\|X_i\|,
\end{eqnarray}
where we omit the high order term $\mathcal{O}\big(\|\tau_i(I-\tau_i{A}^{-1})^{-1}F\|^2\big)$.
Combining (\ref{eq422}) and (\ref{eq423}), we arrive at
\begin{eqnarray*}
\frac{\|X_i-\widetilde{X}_i\|}{\|X_i\|}&\lessapprox&\|A\|~\|F\|\cdot\|(I-\tau_i{A}^{-1})^{-1}\|\\
&=&\|A\|~\|A^{-1}\|~\|(I-\tau_i{A}^{-1})^{-1}\|\cdot\frac{\|F\|}{\|A^{-1}\|},
\end{eqnarray*}
which completes the proof.
\end{proof}

\begin{rem}
%Three remarks are in order. Firstly, as we have pointed out in Section 4.1,
%if $A$ is so large that a direct method for the inverse is prohibitive,
%one can generate the Arnoldi bases from solving linear systems with respect to $A$. Denote by $\{{\bf v}_i\}_{i=1}^{mp}$ the Arnoldi basis vectors,
%we have that
%$$
%\|\widetilde{A}^{-1}{\bf v}_i-A^{-1}{\bf v}_i\|\leq\|\widetilde{A}^{-1}-A^{-1}\|=\|F\|,\quad i=1,2,\ldots,mp,
%$$
%and the analysis is analogous. Secondly,
We notice that
$
\|(I-\tau_i{A}^{-1})^{-1}\|_2=\frac{1}{\sigma_{\min}(I-\tau_i{A}^{-1})},
$
where $\|\cdot\|_2$ is the 2-norm. Therefore, if $\sigma_{\min}(I-\tau_i{A}^{-1})$ is sufficiently large {\rm(}say, $\sigma_{\min}(I-\tau_i{A}^{-1})\geq 10^{-4}${\rm)},
the upper bound of the relative errors
of $\{{X_i}\}_{i=1}^{\nu}$ will be
dominated by the condition number of $A$ and the relative error of $A^{-1}$.
Moreover, we point out that the assumption of $\|\tau_i(I-\tau_i A^{-1})^{-1}F\|_2\ll 1$ is not stringent if $\sigma_{\min}(I-\tau_i{A}^{-1})$ is not small.
Indeed,
a sufficient
condition is that
$
\|F\|_2\ll\frac{\sigma_{\min}(I-\tau_i{A}^{-1})}{|\tau_i|},~i=1,2,\ldots,\nu.
$
\end{rem}

\section{Numerical Experiments}
\setcounter{equation}{0}

In this section, we perform some numerical
experiments and show the superiority of Algorithm 2 over many state-of-the-art algorithms for computing (\ref{1}). All the numerical experiments are run
on a Dell Workstation with four core Intel(R) Pentium(R) processor with CPU
3.2 GHz and RAM 16 GB, under the Windows XP 64 bit operating system. All
the experimental results are obtained from using a MATLAB 7.7 implementation
with machine precision $\epsilon\approx 2.22\times 10^{-16}$.
The algorithms used in this section are listed as follows.\\
%\begin{flushleft}
$\bullet$ {\bf expm} is the MATLAB built-in function for matrix exponential,
which implements the scaling and squaring method \cite{Higham2}.\\
$\bullet$ {\bf SBFOM-DR} and {\bf PSBFOM-DR} are the shifted block FOM algorithm with deflation (Algorithm 1) and the preconditioned and shifted block FOM algorithm with deflation (Algorithm 2), respectively.\\
$\bullet$ {\bf expv} is the
MATLAB function due to Sidje \cite{Sidje}, which evaluates $e^{A}{\bf b}$ using a restarted Krylov subspace method with
a fixed dimension. The MATLAB codes are available from {\it http://www.maths.uq.edu.au/expokit/}.\\
$\bullet$ The MATLAB
function {\bf phipm} of Niesen and Wright \cite{NW} computes the action of linear
combinations of $\phi$-functions on operand vectors. The implementation combines
time stepping with a procedure to adapt the Krylov subspace size. The MATLAB codes are available from {\it http://www1.maths.leeds.ac.uk/\~~jitse/software.html}.\\
$\bullet$ The MATLAB
function {\bf funm\_kryl} is a realization of the Krylov subspace method with deflated restarting for matrix functions \cite{EEG}. Its effect is to ultimately deflate a specific invariant subspace of the matrix which most impedes the convergence of the restarted Arnoldi approximation process. The MATLAB codes are available from {\it http://www.mathe.tu-freiberg.de/\~~guettels/funm\_kryl/}.\\
$\bullet$ The MATLAB function {\bf expmv} is developed for computing $e^AB$ \cite{AH}, where $A$ is an $n\times n$ matrix and
$B$ is $n\times p$ with $p\ll n$. It uses the scaling part of the scaling and squaring method together with a truncated
Taylor series approximation to the exponential. The MATLAB codes are available from {\it http://www.maths.manchester.ac.uk/\~~almohy/papers.html}.

%\end{flushleft}

We run all these MATLAB functions with
their default parameters, and the convergence tolerance in every algorithm is chosen as $tol=10^{-8}$.
In the tables below, we denote by ``CPU" the
CPU time in seconds, and by ``Mat-Vec" the number of matrix-vector products. Let $Z(t)=e^{tA}B$ be the ``exact" solution obtained from running the MATLAB build-in function \texttt{expm}, and let $\widetilde{Z}(t)$ be an approximation computed from other algorithms. Then we make use of
\begin{equation}
{\bf Error}=\frac{\|Z(t)-\widetilde{Z}(t)\|_F}{\|Z(t)\|_F}
\end{equation}
as the relative error of the approximation $\widetilde{Z}(t)$. If an algorithm does not converge within acceptable CPU time (say, 12 hours), or ``Error" is larger than $10^{-6}$, then we declare that the algorithm fails to converge. Efficiency of an algorithm is mainly measured
in terms of ``CPU" and ``Error".

In this section, we choose $\nu=14$ for the Carath\'{e}odory-Fej\'{e}r approximation of the exponential. The right-hand sides are generated by using the MATLAB function $B=\texttt{randn(n,p)}$, which is an $n$-by-$p$ matrix with random entries.
In SBFOM-DR and PSBFOM-DR, we fix the block Arnoldi steps to be $m=30$, and set $k$ to be 30, which is the number of approximate eigenvectors
retained from the previous cycle.
Except for Example 5.4, the matrix-vector product $A^{-1}{\bf w}$ is computed by two steps. First,
we use the MATLAB command $[L,U]=\texttt{lu(A)}$ to evaluate the (sparse) LU factors of the matrix $A$, and then compute
$A^{-1}{\bf w}= U\setminus (L\setminus{\bf w})$, where the backslash ``$\setminus$" is the MATLAB left matrix divide command.
The CPU time of PSBFOM-DR is made up of that for computing LU decomposition, solving the shifted linear systems with multiple right-hand sides, as well as that for forming the approximation $\widetilde{Z}(t)$.

{\bf Example 5.1.}~~This experiment is a variation of the one from
 Al-Mohy and Higham \cite{AH}. There are two test matrices in this example, which are generated by the MATLAB function ``\texttt{gallery}".
The first test matrix is the block tridiagonal matrix $A=-2500\times \texttt{gallery('poisson',99)}$, with $A\in\mathbb{R}^{9801\times 9801}$
being a multiple of the standard
finite difference discretization of the 2D Laplacian. We want to compute $e^{A}B$ with the right-hand sides $B=\texttt{randn(9801,3)}$. The second test matrix is generated by $A=\texttt{gallery('lesp',10000)}$. It returns a $10,000\times 10,000$ matrix with sensitive eigenvalues. We try to compute $e^{A}B$ with the right-hand sides $B=\texttt{randn(10000,5)}$. Tables 5.1 and 5.2 list the numerical results.

{\small
\begin{table}[!h]
\begin{center}
\def\temptablewidth{0.8\textwidth}
{\rule{\temptablewidth}{1pt}}
\begin{tabular*}{\temptablewidth}{@{\extracolsep{\fill}}lccc}
{\bf Algorithm}   &{\bf CPU}  &{\bf Mat-Vec} &{\bf Error}  \\\hline
expm       &565.0     &$-$      &$-$    \\
SBFOM-DR   &10.3      &810     &$4.50\times 10^{-9}$\\
{\bf PSBFOM-DR}   &3.17      &123     &$5.76\times 10^{-10}$ \\
expv    &5.98      &6,944     &$1.39\times 10^{-8}$\\
phipm      &6.42      &4,487     &$1.58\times 10^{-8}$\\
funm\_Kryl   &9.60      &900     &$6.83\times 10^{-7}$\\
expmv   &23.5      &28,553     &$7.39\times 10^{-7}$\\
 \end{tabular*}
 {\rule{\temptablewidth}{1pt}}\\
 \end{center}
 \begin{flushleft}
  {\small {\rm Example 5.1,~Table 5.1:~Numerical results of the algorithms for the computation of $e^{A}B$, where $A=-2500\times\texttt{gallery('poisson',99)}\in\mathbb{R}^{9801\times 9801}$, and $B=\texttt{randn(9801,3)}$.}}
 \end{flushleft}
 \end{table}
}

%{\bf Example 2.}~~ $A=gallery('lesp',n)$, $B=randn(n,5),n=1024,4096,16384$

%{\small
%\begin{table}[!h]
%\tabcolsep 0pt \caption{{\rm Example 5.2:~$n=3000$}} \vspace*{-12pt}
%\begin{center}
%\def\temptablewidth{0.6\textwidth}
%{\rule{\temptablewidth}{1pt}}
%\begin{tabular*}{\temptablewidth}{@{\extracolsep{\fill}}lccc}
%{\bf Algorithm}   &{\bf CPU}  &{\bf Mat-Vec} &{\bf Accuracy}  \\\hline
%expm      &76.9     &$-$      &$-$    \\
%BFOM-DR   &65.0      &2,417     &$4.59\times 10^{-9}$\\
%{\bf PSBFOM-DR}   &8.94     &355     &$1.93\times 10^{-9}$ \\
%Expokit   &270.8      &13,299     &$1.70\times 10^{-10}$\\
%phipm      &237.4      &11,692     &$3.86\times 10^{-9}$\\
%funm\_Kryl   &1341.6      &4,710     &$3.26\times 10^{-7}$\\
%expmv   &590.0      &11,895     &$8.29\times 10^{-8}$\\\hline
% \end{tabular*}
% {\rule{\temptablewidth}{1pt}}\\
% \end{center}
% \begin{flushleft}
%  {\small {\rm Example 5.3:~Numerical results of the six algorithms on
%the $9,845,725\times 9,845,725$ Wb-edu Web matrix.}}
% \end{flushleft}
% \end{table}
%}

{\small
\begin{table}[!h]
\begin{center}
\def\temptablewidth{0.8\textwidth}
{\rule{\temptablewidth}{1pt}}
\begin{tabular*}{\temptablewidth}{@{\extracolsep{\fill}}lccccr}
{\bf Algorithm}   &{\bf CPU}  &{\bf Mat-Vec} &{\bf Error}  \\\hline
expm      &966.5     &$-$      &$-$    \\
SBFOM-DR   &1318.8     &5,419    &$1.47\times 10^{-8}$\\
{\bf PSBFOM-DR}   &97.4     &205     &$3.14\times 10^{-9}$ \\
expv   &8583.5    &38,347    &$5.57\times 10^{-11}$\\
phipm      &5240.4      &23,258     &$8.46\times 10^{-9}$\\
funm\_Kryl  % &1027.8     &3,150     &$2.97\times 10^{-1}$\\
           &f.c. &f.c. &f.c.\\
expmv   &17572.3      &39,225     &$2.73\times 10^{-7}$\\
 \end{tabular*}
 {\rule{\temptablewidth}{1pt}}\\
 \end{center}
 \begin{flushleft}
  {\small {\rm Example 5.1,~Table 5.2:~Numerical results of the algorithms for the computation of $e^AB$, where $A=\texttt{gallery('lesp',10000)}\in\mathbb{R}^{10,000\times 10,000}$, and $B=\texttt{randn(10000,5)}$. Here ``f.c." denotes ``fails to converge".}}
 \end{flushleft}
 \end{table}
}

Some remarks are in order. First, we see that PSBFOM-DR can be applied to large and sparse matrix
exponential successfully. Second, the new algorithm is efficient for the computation of matrix exponential.
It is seen from Tables 5.1 and 5.2 that PSBFOM-DR outperforms the other six algorithms in terms of both the CPU time and
the number of matrix-vector products.
Specifically, PSBFOM-DR converges much faster than SBFOM-DR, 3.17 seconds vs. 10.3 seconds for the first test problem,
and 97.4 seconds vs 1318.8 seconds for the second one.
This illustrates that our preconditioning strategy is very effective for the shifted linear systems with multiple right-hand sides.
Third, the number of matrix-vector products is not the whole story for computing the matrix exponential problem. For example,
we notice from Table 5.1 that
PSBFOM-DR used 123 matrix-vector products and 3.17 seconds, while expv used 6944 matrix-vector products and 5.98 seconds. The reason is that
the CPU time of PSBFOM-DR includes that for the LU decomposition and for solving the shifted linear systems.
Moreover, in the first test problem, phipm uses 4487 matrix-vector products and 6.42 seconds, while expv exploits 6944 matrix-vector products and 5.98 seconds.
As a result, an algorithm using fewer matrix-vector products may not converge faster than another using more matrix-vector products, and vice versa.

{\bf Example 5.2.}~~In this example, we try to show that the preconditioned and shifted block FOM algorithm (PSBFOM-DR) is favorable to exponential of $tA$
with a large norm. The test matrices are two symmetric positive matrix (SPD) matrices that are available from the University of Florida Sparse Matrix Collection: {\it http://www.cise.ufl.edu/ research/sparse/matrices}.
The first test matrix is the {\it 1138bus} matrix arising from power system networks. It is of size $1138\times 1138$, with 2596 nonzero elements.
The second one is the {\it Pres\_Poisson} matrix arising from computational fluid dynamics problems.
The size of this matrix is $14,822\times 14,822$, with 715,804 nonzero elements. We want to compute $e^{t A}B$, with $t=-1,-10,-100$ for the {\it 1138bus} matrix,
and $t=-100,-1000,-10000$ for the {\it Pres\_Poisson} matrix. Tables 5.3 and 5.4 report the numerical results.

 {\small
\begin{table}[!h]
\begin{center}
\def\temptablewidth{0.8\textwidth}
{\rule{\temptablewidth}{1pt}}
\begin{tabular*}{\temptablewidth}{@{\extracolsep{\fill}}lcccc}
{\bf Algorithm}  &$t$ &{\bf CPU}  &{\bf Mat-Vec} &{\bf Error}  \\\hline
expm    &$-1$  &1.31     &$-$      &$-$    \\
SBFOM-DR &$-1$ &5.28     &2,190     &$3.17\times 10^{-12}$\\
{\bf PSBFOM-DR} &$-1$  &1.95     &524     &$4.46\times 10^{-12}$ \\
expv &$-1$  &2.30     &12,214     &$9.51\times 10^{-13}$\\
phipm   &$-1$   &1.59      &7,260     &$3.00\times 10^{-11}$\\
funm\_Kryl  &$-1$ &198.8    &2,370     &$1.05\times 10^{-9}$\\
expmv &$-1$  &9.08     &70,616     &$1.86\times 10^{-10}$\\\hline
expm    &$-10$  &1.45     &$-$      &$-$    \\
SBFOM-DR &$-10$  &180.6     &74,190     &$7.55\times 10^{-12}$\\
{\bf PSBFOM-DR} &$-10$  &0.45     &164     &$4.07\times 10^{-12}$ \\
expv &$-10$  &13.3   &85,002     &$2.93\times 10^{-12}$\\
phipm   &$-10$  &9.67     &38,261    &$2.45\times 10^{-11}$ \\
funm\_Kryl  &$-10$ &f.c.     &f.c.     &f.c.\\
expmv &$-10$  &87.2    &704,275     &$8.08\times 10^{-11}$\\\hline
expm    &$-100$  &1.59     &$-$      &$-$    \\
SBFOM-DR &$-100$  &39.2     &16,050     &$5.38\times 10^{-10}$\\
{\bf PSBFOM-DR} &$-100$  &0.47     &164     &$1.94\times 10^{-11}$ \\
expv &$-100$  &90.5     &565,378    &$1.76\times 10^{-11}$\\
phipm   &$-100$   &44.8      &126,191     &$4.27\times 10^{-11}$\\
funm\_Kryl  &$-100$ &f.c.     &f.c.     &f.c.\\
expmv &$-100$  &1109.3      &7,040,894     &$1.47\times 10^{-11}$\\
 \end{tabular*}
 {\rule{\temptablewidth}{1pt}}
 \end{center}
 \begin{flushleft}
  {\small {\rm Example 5.2,~Table 5.3:~Numerical results of the algorithms on $e^{tA}B$, where $A$ is the $1138\times 1138$ {\it 1138bus} matrix, $t=-1,-10,-100$ and $B=\texttt{randn(1138,4)}$. Here ``f.c." denotes ``fails to converge".}}
 \end{flushleft}
 \end{table}
}

Again, the numerical results illustrate that PSBFOM-DR is superior to the state-of-the-art algorithms for the matrix exponential computation in most cases, especially when
$\|tA\|$ is large.
Moreover,
one observes that the larger $\|tA\|$ is, the less the CPU time is required for PSBFOM-DR. The reason is that
$\frac{\|(tA-\tau_{i}I)-tA\|}{\|tA-\tau_{i}I\|}$ decreases as $t$ increases, refer to (\ref{43}). That is, the larger $\|tA\|$ (or $t$) is,
the better the preconditioner will be.
As a comparison, the CPU time for the other algorithms
increase as $\|tA\|$ (or $t$) becomes large.
Therefore, PSBFOM-DR is preferable to the matrix exponential problem as the coefficient matrix has a large norm or when $t$ is large.

{\small
\begin{table}[!h]
\begin{center}
\def\temptablewidth{0.8\textwidth}
{\rule{\temptablewidth}{1pt}}
\begin{tabular*}{\temptablewidth}{@{\extracolsep{\fill}}lcccc}
{\bf Algorithm}  &$t$ &{\bf CPU}  &{\bf Mat-Vec} &{\bf Error}  \\\hline
expm    &$-100$  &8943.8     &$-$      &$-$    \\
SBFOM-DR &$-100$  &21.7     &630     &$4.21\times 10^{-10}$\\
{\bf PSBFOM-DR} &$-100$  &130.9     &1,503     &$4.73\times 10^{-12}$ \\
expv &$-100$  &11.2     &2,139     &$2.16\times 10^{-13}$\\
phipm   &$-100$   &7.42    &1,287     &$7.99\times 10^{-12}$\\
funm\_Kryl  &$-100$ &11.9   &720    &$1.65\times 10^{-12}$\\
expmv &$-100$  &42.2      &6,411     &$1.60\times 10^{-10}$\\\hline
expm    &$-1000$  &8210.8     &$-$      &$-$    \\
SBFOM-DR &$-1000$  &f.c.     &f.c.     &f.c.\\
{\bf PSBFOM-DR} &$-1000$  &44.5     &423     &$2.22\times 10^{-12}$ \\
expv &$-1000$  &60.4     &12,245     &$1.01\times 10^{-12}$\\
phipm   &$-1000$   &45.7      &6,746     &$1.80\times 10^{-11}$\\
funm\_Kryl  &$-1000$ &684.3     &2,700     &$9.51\times 10^{-10}$\\
expmv &$-1000$  &410.4      &61,585     &$1.80\times 10^{-10}$\\\hline
expm    &$-10,000$  &7517.6     &$-$      &$-$    \\
SBFOM-DR &$-10,000$  &712.3     &19,050     &$2.96\times 10^{-11}$\\
{\bf PSBFOM-DR} &$-10,000$  &23.7     &183     &$1.01\times 10^{-13}$ \\
expv &$-10,000$  &391.6     &80,786     &$8.33\times 10^{-12}$\\
phipm   &$-10,000$   &238.3      &33,712     &$5.73\times 10^{-11}$\\
funm\_Kryl  &$-10,000$ &f.c.     &f.c.     &f.c.\\
expmv &$-10,000$  &4095.0     &613,358     &$4.34\times 10^{-11}$\\
 \end{tabular*}
 {\rule{\temptablewidth}{1pt}}\\
 \end{center}
 \begin{flushleft}
  {\small {\rm Example 5.2,~Table 5.4:~~Numerical results of the algorithms on $e^{tA}B$, where $A$ is the $14,822\times 14,822$ {\it Pres\_Poisson} matrix, $t=-100,-1000,-10000$ and $B=\texttt{randn(14822,3)}$. Here ``f.c." denotes ``fails to converge".}}
 \end{flushleft}
 \end{table}
}

 \begin{center}
   \scalebox{0.55}{\includegraphics{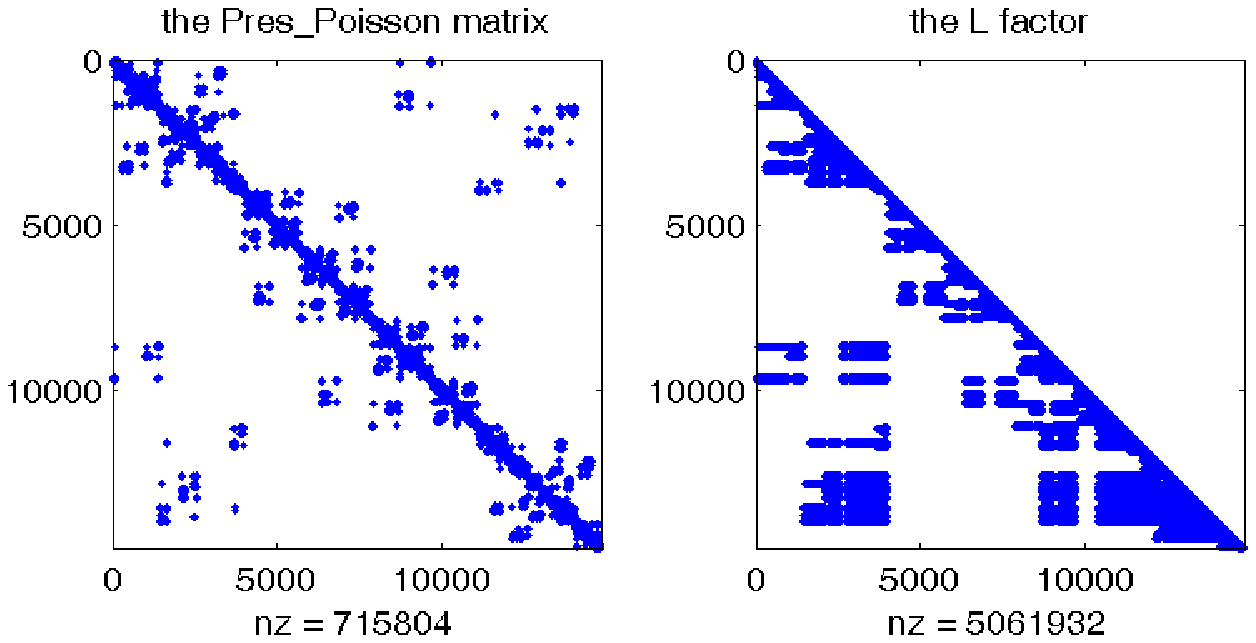}}
 \end{center}
 \begin{flushleft}
 {\small {\rm Example 5.2,~Figure 5.1:~~Sparse structure of the {\it Pres\_Poisson} matrix and that of the $L$ (and $U$)-factor.}}
 \end{flushleft}

Now let's briefly introduce the reason why our preconditioning strategy can be applied to large sparse matrix exponential computations successfully. In order
 to show this more precisely, we plot in Figure 5.1 the sparse structure of the {\it Pres\_Poisson} matrix and that of the $L$-factor (which is also the $U$-factor)
 obtained from the LU factorization of $A$, using the MATLAB command \texttt{spy}. It is observed that the $L$-factor (and the $U$-factor) is still sparse, moreover, the expense for the LU
 factorization is (relatively) much lower than that for computing $e^A$. This explains why the cost for the LU factorization of $A$ is not dominant in the PSBFOM-DR
 algorithm.

On the other hand, we notice that for the {\it Pres\_Poisson} matrix, when $t=-1000$, PSBFOM-DR uses 44.5 seconds and 423 matrix-vector products, while
 {phipm} uses 45.7 seconds and 6746 matrix vector products to achieve the accuracy of $\mathcal{O}(10^{-12})$. That is, it seems that the number of
 matrix-vector products
 used by phipm is about 15.9 times more than PSBFOM-DR, but the CPU time of these two algorithms is about the same; see also Table
 5.3 when $t=-1$. Indeed, as was mentioned in Example 5.1, the CPU time of PSBFOM-DR includes that for LU factorization of $A$.
 Furthermore, this can also be explained by using Figure 5.1, where it shows that the computational cost of performing a matrix-vector product in {PSBFOM-DR} is about 14 times more expensive than that in {phipm}.

{\bf Example 5.3.} This experiment tries to illustrate that PSBFOM-DR still works well even if $A$ is very ill-conditioned. As is known, the matrix exponential plays a fundamental role in solving linear differential equations. In this example, we consider the computation of the product of a matrix exponential with a vector, which arises from the numerical solution of the following fractional diffusion equation \cite{TMS}
\begin{eqnarray}\label{Ex_5.4}
&&\frac{\partial u(x,t)}{\partial t}=d(x)\frac{\partial^{\beta}u(x,t)}{\partial x^{\beta}}+q(x,t), \quad (x,t)\in(0,1)\times(0,1],\\
&& u(0,t)=0, \quad u(1,t)=e^{-t},\quad t\in(0,1],\nonumber\\
&& u(x,0)=x^3, \quad x\in(0,1), \nonumber
\end{eqnarray}
with the coefficient
$$
d(x)=\frac{\Gamma(4-\beta)}{6}x^{1+\beta}
$$
and the source term
$$
q(x,t)=-(1+x)e^{-t}x^3.
$$
Here $1<\beta<2$ and $\Gamma$ is the Gamma function. For the definition of the fractional order derivative, we refer to \cite{P}.

After the spatial discretization by the shifted Gr\"{u}nwald formula \cite{MT}, the equation (\ref{Ex_5.4}) reduces to a semidiscretized ordinary differential equations of the form
\begin{equation}\label{semi_FDE}
\frac{\mathrm{d}\mathbf{u}(t)}{\mathrm{d}t}=A\mathbf{u}(t)+\mathbf{b}(t),\quad \mathbf{u}_0=\mathbf{u}(0),
\end{equation}
where $A=\frac{1}{h^{\beta}}DG$ with $h$ being the grid size, $D$ is a diagonal matrix arising from the discretization of the diffusion coefficient $d(x)$, and $G$ is a lower Hessenberg Toeplitz matrix generated by the discretization of the fractional derivative; see \cite{WWS-2010} for the details of the discretization. The vector $\mathbf{b}(t)=e^{-t}\widetilde{\mathbf{b}}$, where $\widetilde{\mathbf{b}}$ consists of the discretization of $q(x,t)/e^{-t}$ and boundary conditions, and it is independent of $t$. By the variation-of-constants formula, the solution of (\ref{semi_FDE}) at time $t$ can be expressed as
\begin{eqnarray}\label{solution}
\mathbf{u}(t)&=&e^{tA}\mathbf{u}_0+\int_0^te^{(t-\tau)A}e^{-\tau}\mathbf{\widetilde{b}}\mathrm{d}\tau \nonumber\\ &=&e^{tA}(\mathbf{u}_0+(A+I)^{-1}\mathbf{\widetilde{b}})-e^{-t}(A+I)^{-1}\mathbf{\widetilde{b}},
\end{eqnarray}
provided that $A+I$ is invertible. In light of (\ref{solution}), to compute the solution $\mathbf{u}(t)$, we have to approximate the product of the matrix exponential $e^{tA}$ with the vector
$$
\widehat{{\bf b}}=\mathbf{u}_0+(A+I)^{-1}\mathbf{\widetilde{b}},
$$
which is the major computational cost for this problem.

In this experiment, we choose $\beta=1.7,~t=1$, and set the size of the matrix $A$ to be $n=1000,2000$ and 3000, respectively. We mention that the matrix $A$ is very ill-conditioned and it is not a Toeplitz matrix. Indeed, as $n=1000,2000$ and 3000, the 1-norm condition numbers \big(estimated by using the MATLAB command \texttt{condest}\big) are about $3.44\times 10^{10},3.64\times 10^{11}$ and $1.45\times 10^{12}$, respectively. We run the seven algorithms on this problem. Table 5.5 lists the numerical results.

{\small
\begin{table}[!h]
\begin{center}
\def\temptablewidth{0.8\textwidth}
{\rule{\temptablewidth}{1pt}}
\begin{tabular*}{\temptablewidth}{@{\extracolsep{\fill}}lccccr}
{\bf Algorithm} &$n$  &{\bf CPU}  &{\bf Mat-Vec} &{\bf Error}  \\\hline
expm   &1000   &9.61    &$-$      &$-$    \\
SBFOM-DR &1000  &f.c.      &f.c.     &f.c.\\
{\bf PSBFOM-DR} &1000  &1.48  &41     &$7.11\times 10^{-7}$ \\
expv  &1000 &15.0      &8,897     &$1.95\times 10^{-12}$\\
phipm   &1000   &10.2      &5,725     &$4.34\times 10^{-11}$\\
funm\_Kryl &1000 &f.c.      &f.c.     &f.c.\\
expmv  &1000 &241.8      &143,964     &$1.05\times 10^{-9}$\\\hline
expm   &2000   &78.3     &$-$      &$-$    \\
SBFOM-DR &2000  &f.c.     &f.c.     &f.c.\\
{\bf PSBFOM-DR} &2000  &7.66     &41     &$1.93\times 10^{-7}$ \\
expv &2000  &227,0      &25,172     &$1.10\times 10^{-11}$\\
phipm   &2000   &138.0   &14,802     &$1.10\times 10^{-11}$\\
funm\_Kryl &2000  &f.c.     &f.c.     &f.c.\\
expmv &2000  &4196.0  &472,787     &$1.61\times 10^{-9}$\\\hline
expm   &3000   &257.0   &$-$      &$-$    \\
SBFOM-DR &3000  &f.c.   &f.c.     &f.c.\\
{\bf PSBFOM-DR} &3000  &27.3     &41     &$2.47\times 10^{-8}$ \\
expv &3000  &1008.1   &50,437    &$1.84\times 10^{-11}$\\
phipm   &3000   &505.1  &24,458    &$1.79\times 10^{-11}$\\
funm\_Kryl &3000 &f.c.     &f.c.     &f.c.\\
expmv  &3000 &18749.3  &946,657 &$2.05\times 10^{-9}$\\
 \end{tabular*}
 {\rule{\temptablewidth}{1pt}}\\
 \end{center}
 \begin{flushleft}
  {\small {\rm Example 5.3,~Table 5.5, :~Numerical results of Example 5.3 for computing $e^A\widehat{\bf b}$ with $\beta=1.7$ and $n=1000,2000$ and 3000. The 1-norm condition numbers \big(estimated by using the MATLAB command \texttt{condest}\big) are about $3.44\times 10^{10},3.64\times 10^{11}$ and $1.45\times 10^{12}$, respectively. Here ``f.c." denotes ``fails to converge".}}
 \end{flushleft}
 \end{table}
}

We see from Table 5.5 that our new algorithm performs much better than the other algorithms in terms of CPU time and number of matrix-vector products, moreover, it can still reach an acceptable accuracy even if $A$ is very ill-conditioned. So
our new algorithm is promising even if the matrix in question is very ill-conditioned. However, it is seen that the accuracy of our approximation is lower than those obtained from the other algorithms. This can be interpreted by using Theorem 4.3, where it is shown
that the error of the computed solution is affected by  the ill-conditioning of the matrix $A$ in question.

{\bf Example 5.4.}~~In this example, we consider approximation of $e^{tA}B$ with $A$ being a Toeplitz matrix. Toeplitz matrices arise from numerous topics like signal and image processing,
numerical solutions of partial differential equations and integral equations, queueing
networks \cite{CNG,CJ}, and so on. The Toeplitz matrix exponential problem plays an important role in various application fields such as computational finance \cite{LPS,PS,TGB}. Moreover, in integral equations, the Toeplitz matrix exponential also takes part in the numerical
solution of Volterra-Wiener-Hopf equations \cite{AB}. However, Toeplitz matrices generally are dense, and some classic methods
for approximating the Toeplitz matrix exponential will suffer from $\mathcal{O}(n^3)$ complexities \cite{ML2}.

Based on the shift-and-invert Arnoldi method \cite{MN,vH}, Toeplitz structure and the famous Gohberg-Semencul
formula (GSF) \cite{GSF}, Lee, Pang and Sun \cite{LPS} proposed a shift-and-invert Arnoldi algorithm for Toeplitz matrix exponential, which can reduce the computational cost to $\mathcal{O}(n \log n)$ in total. However, there are some deficiencies in this algorithm. For instance, this algorithm is a non-restarted one and there is no {\it posteriori} stopping criterion available, so one does not know how to restart and when to terminate this algorithm properly. Furthermore, we have no idea of choosing the {\it optimal} shift in the shift-and-invert Arnoldi algorithm in advance, if there is no other information {\it a prior}.

The aim of the example is twofold. First, we compare our preconditioned and shifted block FOM algorithm (PSBFOM-DR) with the shift-and-invert Arnoldi method for Toeplitz matrix exponential \cite{LPS}, and show the superiority of the former over the latter. Second, we demonstrate that {PSBFOM-DR} is feasible for matrix exponential of very large matrices, provided that the inverse of the matrix $A$ can be computed efficiently. The test matrix $A$ comes from the spatial discretization of the following fractional diffusion equation by the shifted Gr\"{u}nwald formula \cite{MT}
\begin{eqnarray}\label{Ex_5.6}
&& \frac{\partial u(x,t)}{\partial
t}=d_1(x)\frac{\partial^{\beta}u(x,t)}{\partial_+
x^{\beta}}+d_2(x)\frac{\partial^{\beta}u(x,t)}{\partial_-
x^{\beta}}+q(x,t),~ x\in(0,1),~1<\beta<2,
\end{eqnarray}
where $d_1(x)=1$, $d_2(x)=3$, and $u(0,t)=u(1,t)=0$. We refer to \cite{P} for the definition of the fractional order derivatives. As the diffusion coefficients are constant, the resulting matrix after the spatial discretization is a Toeplitz matrix \cite{WWS-2010}.

 {\small
\begin{table}[!h]
\begin{center}
\def\temptablewidth{0.8\textwidth}
{\rule{\temptablewidth}{1pt}}
\begin{tabular*}{\temptablewidth}{@{\extracolsep{\fill}}lcccc}
{\bf Algorithm} &$n$  &{\bf CPU}  &{\bf Mat-Vec} &{\bf Error}  \\\hline
{\bf PSBFOM-DR} &$5\times 10^4$  &27.5  &82     &$6.74\times 10^{-8}$ \\
SI-Arnoldi  &$5\times 10^4$    &28.3    &120     &$5.09\times 10^{-8}$\\
SI-BArnoldi &$5\times 10^4$   &46.5     &120     &$4.45\times 10^{-8}$\\
expv        &$5\times 10^4$   &f.c.     &f.c,     &f.c.\\
phipm       &$5\times 10^4$   &f.c.     &f.c,     &f.c.\\
funm\_Kryl  &$5\times 10^4$   &f.c.     &f.c.     &f.c.\\
expmv       &$5\times 10^4$   &f.c.     &f.c,     &f.c.\\
\hline
{\bf PSBFOM-DR} &$1\times 10^5$  &77.6  &82     &$7.51\times 10^{-8}$ \\
SI-Arnoldi  &$1\times 10^5$    &89.2    &120     &$1.11\times 10^{-7}$\\
SI-BArnoldi &$1\times 10^5$   &119.8     &120     &$1.18\times 10^{-7}$\\
expv        &$1\times 10^5$   &f.c.     &f.c,     &f.c.\\
phipm       &$1\times 10^5$   &f.c.     &f.c,     &f.c.\\
funm\_Kryl  &$1\times 10^5$   &f.c.     &f.c.     &f.c.\\
expmv       &$1\times 10^5$   &f.c.     &f.c,     &f.c.\\
\hline
{\bf PSBFOM-DR} &$1.5\times 10^5$  &148.6  &82     &$5.49\times 10^{-7}$ \\
SI-Arnoldi  &$1.5\times 10^5$    &169.1    &120     &$2.40\times 10^{-7}$\\
SI-BArnoldi &$1.5\times10^5$   &227.6     &120     &$2.25\times 10^{-7}$\\
expv        &$1.5\times10^5$   &f.c.     &f.c,     &f.c.\\
phipm       &$1.5\times10^5$   &f.c.     &f.c,     &f.c.\\
funm\_Kryl  &$1.5\times10^5$   &f.c.     &f.c.     &f.c.\\
expmv       &$1.5\times10^5$   &f.c.     &f.c,     &f.c.\\
\hline
{\bf PSBFOM-DR} &$2\times 10^5$  &227.9  &82     &$1.75\times 10^{-7}$ \\
SI-Arnoldi  &$2\times 10^5$    &261.0    &120     &$4.10\times 10^{-8}$\\
SI-BArnoldi &$2\times 10^5$   &345.8     &120     &$3.71\times 10^{-8}$\\
expv        &$2\times 10^5$   &f.c.     &f.c,     &f.c.\\
phipm       &$2\times 10^5$   &f.c.     &f.c,     &f.c.\\
funm\_Kryl  &$2\times 10^5$   &f.c.     &f.c.     &f.c.\\
expmv       &$2\times 10^5$   &f.c.     &f.c,     &f.c.\\
 \end{tabular*}
 {\rule{\temptablewidth}{1pt}}\\
 \end{center}
 \begin{flushleft}
  {\small {\rm Example 5.4,~Table 5.6:~Numerical results of the seven algorithms for Toeplitz matrix exponential, $\beta=1.8,~B=\texttt{randn(n,2)}$.}}
 \end{flushleft}
 \end{table}
}

We run seven algorithms: PSBFOM-DR, expv, phipm, funm\_Kryl, expmv, the shift-and-invert Arnoldi algorithm ({SI-Arnoldi}) \cite{LPS}, and a shift-and-invert block Arnoldi algorithm ({SI-BArnoldi}, which is a generalization of {SI-Arnoldi} to the block case) on the matrices with size $n=5\times 10^4,10^5,1.5\times 10^5$ and $2\times 10^5$, respectively. Since $A$ is a Toeplitz matrix, we can use the Gohberg--Semencul formula \cite{GSF} for the inverse of Toeplitz matrix in PSBFOM-DR, SI-Arnoldi and SI-BArnoldi, which can be realized by using only six FFTs of length $n$ \cite{LPS,PS}. Note that one only needs to solve the linear systems only once for the Toeplitz inverse \cite{LPS,PS}. For the sake of justification, we replace the matrix-vector products in PSBFOM-DR, expv, phipm, funm\_Kryl, and expmv by the fast Toeplitz matrix-vector products \cite{CNG,CJ}.

In this example, we take $\beta=1.8$ and $t=1$ and compute $e^{tA}B$ with $B$ being generated by $\texttt{randn(n,2)}$. As was done in \cite{LPS}, we pick the shift $\gamma=1/10$ in the shift-and-invert Arnoldi algorithm and the shift-and-invert block Arnoldi algorithm. As the Toplitz matrices
are very large, the MATLAB build-in function \texttt{expm} is infeasible for this problem. As a compromise, we set the approximations got from running the shift-and-invert Arnoldi algorithm with the step $m=100$ to be the ``exact" solutions.

We use $m=30$ for the PSBFOM-DR algorithm. As {SI-Arnoldi} is a non-restarted and non-block algorithm, we use the step $m=60$ for {SI-Arnoldi}, so that the dimension of the search subspaces of PSBFOM-DR and {SI-Arnoldi} are the same. The number of the block shift-and-invert Arnoldi steps $m$ is also chosen as 60 in {SI-BArnoldi}. Table 5.6 lists the numerical results.
It is observed from this table that all the algorithms {expv}, phipm, {funm\_Kryl}, and expmv fail to converge for this example.
One reason is that the norm of the Toeplitz matrix $A$ can be very large when $n$ is large. For instance, when $n=5\times 10^4,10^5,1.5\times 10^5$ or $2\times 10^5$, the norm of $A$ is in the order of $\mathcal{O}(10^9)$ or even $\mathcal{O}(10^{10})$. It is shown that the accuracy and convergence of those methods are closely related to the norm of $A$ (refer to Example 5.2), and a large norm may lead to very slow convergence and even bad accuracy.
%Second, when $n$ is large, the elements of $A$ can also be  very large. Then implementing the fast Toepltiz matrix vector product by FFTs \cite{CJ} may lead to a certain rounding errors. Thus, the matrix-vector products can not be computed accurately enough which may ultimately affect the accuracy and convergence of those methods.
As a comparison, PSBFOM-DR works quite well for this problem. Moreover, it converges faster than {SI-Arnoldi} and {SI-BArnoldi} in many cases, and the accuracy of the solutions obtained from Algorithm 2 is comparable to those obtained from the two shift-and-invert algorithms. These demonstrate the potential of PSBFOM-DR for computing exponential of
very large Toeplitz matrices.

\section{Concluding remarks}
%For a given large sparse matrix $A$ and a rectangular matrix $B$, the problem of numerical approximating $e^{tA}B$ arises in many applications. In many practical situations, the matrix exponential is often replaced by a rational function, and the computation of this problem
%reduces to solving linear
%systems with multiple shifts and multiple right-hand sides. The use of standard Krylov subspace techniques in this context has been
%actively investigated, however, these methods often turn out to be inefficient because of slow
%convergence.

In this paper, we propose a preconditioned and shifted block FOM algorithm with deflation for computing the matrix exponential problem.
The key idea for the preconditioning technique is that the absolute values of the poles of
the Carath$\acute{\rm e}$odory-Fej$\acute{\rm e}$r approximation can be much smaller than the norm of the matrix in question.
The new method can precondition all the shifted linear systems simultaneously, and preserve the original
structure of the shifted linear systems when restarting. Furthermore, when both $A$ and $B$ are real while the shifts are complex, the expensive step of constructing the orthogonal basis can be realized in real arithmetics.

The new algorithm can also be applied to many other problems which reduces to solving shifted linear systems with multiple right-hand
sides. For instance, in the quantum chromodynamics (QCD) problems, it is common to solve linear systems with multiple shifts and multiple right-hand
sides \cite{BF,DMW,Vorst}.
In the Wilson-Dirac problems, the right-hand sides
represent different noise vectors, and the shifts $\{\tau_i\}_{i=1}^{\nu}$
correspond to different quark masses that are used in an
extrapolation \cite{DMW}. Consequently, the computation of the sign
function of a large matrix $A$ resorts to the problem of linear
system which have both multiple right-hand sides and multiple
complex shifts \cite{BF,STK}.
In future work, inexact and restarted algorithms would be studied for solving linear systems with $A$ in the inner iterations, where the matrix-vector products with respect to $A^{-1}$ can be relaxed as the residual of the outer iteration approaches zero \cite{BF1}.
It is a very interesting topic and deserves further investigation.

\section*{Acknowledgments}
The first author would like to thank Prof. Xiao-qing Jin, Prof. Hai-wei Sun and Dr. Siu-long Lei for their helpful discussions during his visiting of University of Macau.

\end{document}